\documentclass[leqno,11pt,oneside]{amsart}

\usepackage{geometry}
\usepackage{physics}
\usepackage{times}
\usepackage{enumerate}

\usepackage[all]{xy} % diagrams

\usepackage{amsmath, amssymb, amsfonts, latexsym, mdwlist, amsthm,amscd}
\usepackage{caption, subcaption}
\usepackage{tikz}
\usetikzlibrary{calc,trees,arrows,chains,shapes.geometric,%
	decorations.pathreplacing,decorations.pathmorphing,shapes,%
	matrix,shapes.symbols}

\tikzset{
	>=stealth',
	punktchain/.style={
		rectangle,
		rounded corners,
		draw=black, thick,
		minimum height=3em,
		text centered,
		on chain},
	line/.style={draw, thick, <-},
	element/.style={
		tape,
		top color=white,
		bottom color=blue!50!black!60!,
		minimum width=8em,
		draw=blue!40!black!90, very thick,
		text width=10em,
		minimum height=3.5em,
		text centered,
		on chain},
	every join/.style={->, thick,shorten >=1pt},
	decoration={brace},
	tuborg/.style={decorate},
	tubnode/.style={midway, right=2pt},
}

\usepackage{cite}
\usepackage{url}

\usepackage{verbatim}
\usepackage{mathtools}

\usepackage{pgf,tikz}
\usepackage{tikz-cd}
\usetikzlibrary{calc,matrix,patterns,shadings,backgrounds,fit}
\pgfdeclarelayer{myback}
\pgfsetlayers{myback,background,main}
\usepackage{stmaryrd}
\usepackage{diagbox}
\usepackage{extarrows}

\numberwithin{equation}{section} % numbering equation with section

\usepackage[bookmarks, colorlinks, breaklinks, pdftitle={contractibility}]{hyperref}
\hypersetup{linkcolor=blue,citecolor=blue,filecolor=black,urlcolor=blue}

\usepackage{paralist}
\setdefaultenum{(a)}{(i)}{}{}
\def\C{\ensuremath{\mathbb{C}}}

\def\P{\ensuremath{\mathbb{P}}}
\def\Q{\ensuremath{\mathbb{Q}}}
\def\R{\ensuremath{\mathbb{R}}}
\def\Z{\ensuremath{\mathbb{Z}}}

\def\Amp{\mathop{\mathrm{Amp}}\nolimits}

\def\ch{\mathop{\mathrm{ch}}\nolimits}

\def\Coh{\mathop{\mathrm{Coh}}\nolimits}

\def\dim{\mathop{\mathrm{dim}}\nolimits}
\def\ev{\mathop{\mathsf{ev}}\nolimits}

\def\Eff{\mathop{\mathrm{Eff}}}

\def\Ext{\mathop{\mathrm{Ext}}\nolimits}

\def\Hom{\mathop{\mathrm{Hom}}\nolimits}
\def\id{\mathop{\mathsf{id}}\nolimits}

\def\mod{\mathop{\mathrm{mod}}\nolimits}
\def\min{\mathop{\mathrm{min}}\nolimits}

\def\Nef{\mathop{\mathrm{Nef}}\nolimits}
\def\NS{\mathop{\mathrm{NS}}\nolimits}

\def\Pic{\mathop{\mathrm{Pic}}}

\def\rk{\mathop{\mathrm{rk}}}

\def\Spec{\mathop{\mathrm{Spec}}}

\def\Db{\mathrm{D}^{b}}

\def\Cone{\mathrm{Cone}}

\makeatletter
\newtheorem*{rep@theorem}{\rep@title}
\newcommand{\newreptheorem}[2]{%
\newenvironment{rep#1}[1]{%
 \def\rep@title{#2 \ref{##1}}%
 \begin{rep@theorem}}%
 {\end{rep@theorem}}}
\makeatother

\newtheorem{Thm}{Theorem}[section]
\newreptheorem{Thm}{Theorem}
\newtheorem{Prop}[Thm]{Proposition}
\newtheorem{PropDef}[Thm]{Proposition and Definition}
\newtheorem{Lem}[Thm]{Lemma}

\newtheorem{Cor}[Thm]{Corollary}
\newreptheorem{Cor}{Corollary}

\newtheorem{Ques}[Thm]{Question}

\newtheorem{thm-int}{Theorem}

\theoremstyle{definition}
\newtheorem{Def-s}[Thm]{Definition}
\newtheorem{Def}[Thm]{Definition}
\newtheorem{Rem}[Thm]{Remark}

\newtheorem{Ex}[Thm]{Example}

\newtheorem{Not}[Thm]{Notation}

\def\C{\ensuremath{\mathbb{C}}}

\def\P{\ensuremath{\mathbb{P}}}
\def\Q{\ensuremath{\mathbb{Q}}}
\def\R{\ensuremath{\mathbb{R}}}
\def\Z{\ensuremath{\mathbb{Z}}}

\def\cA{\ensuremath{\mathcal A}}

\def\cI{\ensuremath{\mathcal I}}

\def\cL{\ensuremath{\mathcal L}}

\def\cO{\ensuremath{\mathcal O}}

\def\cW{\ensuremath{\mathcal W}}

\def\iff{\; \Longleftrightarrow \;}

\usepackage{todonotes}

\def\RHom{\mathrm{RHom}}

\def\Nef{\mathrm{Nef}}
\def\Peff{\overline{\mathrm{NE}}}
\def\NS{\mathrm{NS}}
\def\Amp{\mathrm{Amp}}
\def\cud{\mathbf{d}}
\def\cuh{\mathbf{h}}
\def\cuc{\mathbf{c}}
\begin{document}

\title[Spherical bundles]{A note on spherical bundles on K3 surfaces}

\author{Chunyi Li}
\address{C. L.:
Mathematics Institute, University of Warwick,
Coventry, CV4 7AL,
United Kingdom}
\email{C.Li.25@warwick.ac.uk}
\urladdr{https://sites.google.com/site/chunyili0401/}

\author{Shengxuan Liu}
\address{S. L.:
Mathematics Institute, University of Warwick,
Coventry, CV4 7AL,
United Kingdom}
\email{Shengxuan.Liu.1@warwick.ac.uk}
\urladdr{https://warwick.ac.uk/fac/sci/maths/people/staff/sliu/}

\address{G.O.:
Graduate School of Mathematics, Nagoya University, 
Furocho, Chikusaku, Nagoya, 464-8602, Japan}
\email{genki.ouchi@math.nagoya-u.ac.jp}
\urladdr{https://www.math.nagoya-u.ac.jp/en/people/faculty-05.html}

%\date{\today}
%\subjclass[2020]{14F08}

\begin{abstract} 

Some questions are posted at the end of \cite[Chapter 16]{HuybrechtsK3book}, concerning the bounded derived category of a K3 surface $\Db(S)$. Let $E$ be a spherical object in $\Db(S)$. The first question asks if there always exists a non-zero object $F$ satisfying $\RHom(E,F)=0$. Further, let $E$ be a spherical bundle. The second question is whether $E$ is always semistable with respect to some polarization on $S$ and if there is a way to `count'  spherical bundles with a fixed Mukai vector. In this note, we provide (partial) answers to these two questions. In the appendix, Genki Ouchi shows that any spherical twist associated to an $n$-spherical object on a smooth projective $n$-dimensional variety is not conjugate to a standard autoequivalence.
\end{abstract}

\maketitle

\setcounter{tocdepth}{1}
\tableofcontents

\section{Introduction}
Let $S$ be a smooth K3 surface and $E$ be a spherical object in $\Db(S)$, in other words,
\begin{align*}
    \dim\Hom(E,E[i])=\begin{cases}
        1 & \text{ when }i=0,2;\\
        0 & \text{ otherwise. }
    \end{cases}
\end{align*}
The first question at the end of \cite[Chapter 16]{HuybrechtsK3book} asks whether $E^\perp\neq 0$, in other words, does there exist a non-trivial object $F$ such that $\Hom(E,F[i])=0$ for all $i\in \Z$.

In the case that the K3 surface $S$ is of Picard number one, the question is answered in \cite[Apendix A]{genkientropy} by Bayer using tools and results from \cite{K3Pic1}.  Here, we provide an affirmative answer to this question for all K3 surfaces and some more general cases.

\begin{Prop}[{Proposition \ref{prop:sphericalorth}}]\label{thm:mainEperp}
    Let $X$ be a smooth proper $n$-dimensional variety over $\C$ and $E$ be an $n$-spherical object in $\Db(X)$. Then there exists non-trivial object $F$ orthogonal to $E$, in other words, $\Hom(E,F[k])=\Hom(F,E[k])=0$ for every $k\in \Z$. 
\end{Prop}
Proposition \ref{thm:mainEperp} has some immediate corollaries. For example, \cite[Theorem 3.1]{genkientropy} now holds for all spherical objects in $\Db(X)$ without any extra assumptions. Another corollary is pointed out with this question in \cite[Chapter 16]{HuybrechtsK3book}, the composition of spherical twist $\mathsf T^k_E$ is never a homological shift when $k\neq 0$. In fact, the autoequivalence $\mathsf{T}^k_E$ is not conjugate to a standard autoequivalence with respect to $X$ in $\mathrm{Aut}(\Db(X))$. We refer to Theorem \ref{thm:non-standard} in the Appendix for more details.

We may ask a consecutive question about $E^\perp$.
\begin{Ques}\label{ques:nodalsg}
    Let $S$ be a smooth projective surface and $E$ be a $2$-spherical object in $\Db(S)$. Denote by $\langle E^\perp, E\rangle$ the full sub-triangulated category generated by $E$ and objects in $E^\perp$.  Do we always have the equivalence \begin{align*}
        \Db(S)/\langle E^\perp,E\rangle\cong \mathrm{D}_{\mathrm{sing}}(\Spec(\C[x]/(x^2)))?
    \end{align*}
\end{Ques}
A few more details regarding this question will be discussed in Remark \ref{rem:sphericalperp}. Moving on to the next question, consider a smooth projective K3 surface $S$ with Picard number greater than $1$, and let $E$ be a spherical bundle on $S$. The second question at the end of \cite[Chapter 16]{HuybrechtsK3book} asks if $E$ is always Gieseker stable with respect to some polarization, and if there is a way to `count' spherical bundles on $S$ with a given Mukai vector. Definitions on stability and Mukai vector will be recapped at the beginning of Section \ref{sec3}.

Unlike the first question, the naive answer to the above question is negative. 
\begin{Ex}\label{eg:main}
 \begin{enumerate}
        \item We give an example of a K3 surface $S$ of Picard number $2$ and a spherical  bundle $F$ on $S$. The bundle $F$ is not slope semistable (in particular, not Gieseker semistable) with respect to any polarization of $S$.
        \item We give an example of a K3 surface $S$ of Picard number $3$ and a Mukai vector $v=(2,D,c)$ satisfying $\langle v,v\rangle=-2$. We give examples of infinitely many spherical bundles $\{E_i\}$ with  $v(E_i)=v$. Each bundle $E_i$ is slope stable (in particular, Gieseker stable) with respect to some different polarization.
    \end{enumerate}
\end{Ex}

More details of these examples will be given in Example \ref{eg:nonstablesphericalbundle} and \ref{eg:infinitelymanyspherical}. Here we point out that in Example \ref{eg:main} (b), the nef cone of the surface $S$ is circular. This is the main reason that the moduli space of a given numerical class may have infinitely many walls and chambers. Conversely, when the nef cone of the surface is rational polyhedral, this can never happen. More precisely, we have the following result in the more general case.

\begin{Prop}[{Proposition \ref{prop:finwalls}}]\label{thm:mainfinitewall}
    Let $S$ be a smooth projective surface over $\C$ with a rational polyhedral nef cone. Let  $v=(\rk,\ell,s)\in\tilde{\mathrm{H}}(S,\Z)$ with $\rk>0$ and $(\rk,\ell)$ being primitive. Then the moduli space $M^s(v)$ has only finitely many walls and chambers in $\mathrm{Amp}(S)$.
\end{Prop}

Note that for any given polarization $H$ and spherical Mukai vector $v$, there is at most one Gieseker $p_H$-stable (and slope $\mu_H$-stable) vector bundle.  Due to Proposition \ref{thm:mainfinitewall}, the finite theory on walls and chambers, we can talk about `counting' stable spherical bundles in the following sense for a large class of K3 surfaces.
\begin{Cor}\label{cor:maincounting}
    Let $S$ be a smooth projective K3 surface with $\Nef(S)$ being rational polyhedral and $v$ be a spherical Mukai vector  with rank $>0$, then
    \begin{align*}
        \mathsf{H}_S(v):=\#\{E\mid v(E)=v; E\text{ is $\mu_H$-stable with respect to some polarization $H$.}\}<+\infty.
    \end{align*}
\end{Cor}
We then show that the counting can be done purely numerically. In particular, the function $\mathsf H_S(-)$ is completely determined by the numerical information $(\Nef(S), \langle-,-\rangle_{\text{intersection}})$ of the surface $S$.  In Section \ref{sec5}, we make a case study on a generic elliptic K3 surface admitting a section. To convince the reader that such a counting theory may be not completely pointless, we highlight some interesting observations and results on $\mathsf H_S(-)$ when $S$ is an elliptic K3 surface case.

Let $S$ be a generic elliptic K3 surface admitting a section. Assume that $\NS(S)=\Z e\oplus \Z\sigma$, where $e$ stands for the class of the fiber and $\sigma$ for the class of the $(-2)$-curve.  Then for some special spherical Mukai vectors, the value of $\mathsf H_S(v)$ can be explicitly computed and verified. In Example \ref{eg:computeH}, we show that $\mathsf H_S\left((\rk,\sigma,0)\right)=\rk$. In some other cases, the counting seems to have some pattern, but we cannot give proof. For example, let $$b_1=b_2=1\text{ and } b_{n+1}:=b_n+b_{n-1}$$ be the Fibonacci sequence, then $$\mathsf H_S((b_{n+1},b_n\sigma+*e,*))=\left\lfloor\frac{n}2\right\rfloor^2+2,$$ where $*$'s are integers making the Mukai vector spherical.

 In general, it seems unrealistic to give an explicit formula of $\mathsf H_S(v)$ for all $v$. The asymptotic behavior of $\mathsf H_S$ with respect to the rank is already rather difficult.  We have the following expectations on some of the asymptotic behaviors.
 \begin{Ques}\label{ques:asymptofcounting}
    Let $S$ be a generic elliptic K3 surface admitting a section. Is it true that \begin{align*}
        \min_{\rk(v)=m}\{\mathsf H_S(v)\}= \Theta\left((\ln m)^2\right)?\\
        \mathrm{Average}_{\rk(v)=m}\{\mathsf H_S(v)\}= \Theta\left((\ln m)^2\right)?
    \end{align*}
    
 \end{Ques}
   Here we write $f(m)=\Theta(g(m))$ if there exist positive numbers $C_1$, $C_2$, and $N$ such that $C_1g(m)<f(m)<C_2g(m)$ whenever $m>N$. 

 Note that $\mathsf H_S((\rk,\sigma,0))=\rk$, the spherical Mukai character $(\rk,\sigma,0)$ is an outlier. In general, a spherical Mukai vector $v=(m,a\sigma+*e,*)$ is with `$\mathsf H(v)\gg (\ln m)^2$' only when $\tfrac{a}m$ is `very close' to a rational number $\tfrac{s}{r}$ with `$r\ll m$'. For example, the ratio $\tfrac{1}{\rk}$ is very close to $\tfrac{0}{1}$. Such outlier rational numbers consist of a tiny proportion of rational numbers, which tends to $0$ as $m$ tends to infinity.

As a piece of evidence for  Question \ref{ques:asymptofcounting}, we prove the following weaker form on the average estimation for $\mathsf H_S$.
\begin{Prop}[Proposition \ref{prop:estimationofH}]\label{prop:averestiinintro}
      Let $S$ be a generic elliptic K3 surface admitting a section. Then for any $\alpha>0$, there are positive constants $C_1$ and $C_2$ such that for $m\gg0$, \begin{align*}
         C_1\tfrac{\varphi(m)}m(\ln m)^2\leq \mathrm{Average}_{\rk(v)=m}\{\mathsf H_S(v)\}\leq C_2 m^{\alpha}.
    \end{align*} 
\end{Prop}
To get the asymptotic behavior of $\mathrm{Average}_{\rk(v)=m}\{\mathsf H_S(v)\}$, we have reduced the main difficulty to the following analytic number theory question, which might hold independent interests.
\begin{Ques}\label{ques:averofdivfunction}
    For $m\in\Z_{\geq 2}$, let $A_{m}:=\{m^2+r^2-trm\in \Z\mid r,t\in \Z_{\geq 1}, \gcd(m,r)=1\}\cap [1,m^2]$
    and $\mathsf G'(m):=\sum_{a\in A_{m}} \tau(a)$, where $\tau(n)=\sum_{d\vert n} 1$ is the divisor function. What is the asymptotic behavior of $\mathsf G'(m)$?
\end{Ques}

\subsection*{Acknowledgement}
The author wishes to thank Arend Bayer and Xiaolei Zhao for their useful comments. Part of the paper is written during the junior trimester program of  S. Liu and G. Ouchi, funded by the Deutsche Forschungsgemeinschaft (DFG, German Research Foundation) under Germany's Excellence Strategy – EXC-2047/1 – 390685813. We would like to thank the Hausdorff Research Institute for Mathematics for their hospitality. C. Li and S. Liu are supported by the Royal Society URF$\backslash$R1$\backslash$201129 “Stability condition and application in algebraic geometry”.

\textbf{Notations:} We  adopt the standard notions from  \cite{HuybrechtsK3book} in general and we will always work over $\C$. We will recap some notions at the beginning of each section for the convenience of the readers.

For a  smooth proper variety $X$, we will denote by $\Db(X)$ the bounded derived category of coherent sheaves on $X$. Given two objects $E$ and $F$ in $\Db(X)$, we will write $\hom(E,F):=\dim \Hom(E,F)$. Let $S$ be a smooth projective surface over $\C$. We adopt the following standard notions on the cone of $S$: N\'eron--Severi   group  $\NS(S):=\Pic(S)/\Pic^0(S)$ and $\NS(S)_{\R}:=\NS(S)\otimes_{\Z}\R$;  ample cone $\Amp(S)$; nef cone $\Nef(S)=\overline{\Amp}(S)$; effective cone $\Eff(S)$; and  Mori cone $\Peff(S)=\overline{\Eff}(S)$.

\section{Orthogonal component to a spherical object}
\begin{Def}\label{def:spherical}

Let $X$ be a smooth proper $n$-dimensional variety over $\C$ ($n\geq 1)$. A non-zero object $E$ in $\Db(X)$ is called \textit{$n$-spherical} if $\mathsf{S}_X(E)=E[n]$, and
\begin{align*}
    \Hom(E,E[i])=\begin{cases}
        \C & \text{ when $i=0,n$;}\\
        0 & \text{ otherwise.}
    \end{cases}
\end{align*}
    An object $G$ is defined to be in $E^\perp$ if $\RHom(E,G)=0$. 
\end{Def}
\begin{Ex}[an element in $\cO_X^\perp$]\label{eg:spherical}
    Assume that $X$ is a strict Calabi--Yau variety of dimension $n$, in other words, $\omega_X\cong \cO_X$ and $H^i(\cO_X)=0$ when $1\leq i\leq n-1$, then the structure sheaf $\cO_X$ is $n$-spherical. 
    
    Let $p$ and $q$ be two different closed points on $X$. Denote by $\cO_p$ and $\cI_q$ the skyscraper sheaf at $p$ and the ideal sheaf of $q$ respectively. By Serre duality, $\Hom(\cO_p,\cI_q[n])=(\Hom(\cI_q,\cO_p))^*=\C.$ Let $f$ be a non-zero morphism in $\Hom(\cO_p,\cI_q[n])$. We claim that the object $$F:=\Cone(\cO_p\xrightarrow{f}\cI_q[n])$$
    is in $\cO_X^\perp$.
    \end{Ex}
 \begin{proof} Apply $\Hom(\cO_X,-)$ to the distinguished triangle
    \begin{align}
        \cO_p\xrightarrow{f}\cI_q[n]\to F\xrightarrow{+}. \label{eq:2.1}
    \end{align} We get the long exact sequence:
    \begin{align*}
        \dots\to& \Hom(\cO_X,\cO_p[i])\to \Hom(\cO_X,\cI_q[n+i])\to \Hom(\cO_X,F[i])\\\to &\Hom(\cO_X,\cO_p[i+1])\to \dots.
    \end{align*}
    Note that $\RHom(\cO_X,\cI_q[n])=\C$ and $\RHom(\cO_X,\cO_p)=\C$, we get $\Hom(\cO_X,F[i])=0$ when $i\neq 0,-1$. Moreover, the spaces
    \begin{align}\label{eq2112}
        \Hom(\cO_X,F)\cong \Hom(\cO_X,F[-1]),\text{ and they are both $0$ or $\C$.}
    \end{align}

    Apply $\Hom(-,\cO_X)$ to \eqref{eq:2.1}, we get the long exact sequence
    \begin{align}\label{eq:2.111}
        \dots\to& \Hom(\cO_p,\cO_X[n-1])=0\to \Hom(F,\cO_X[n])\to \Hom(\cI_q[n],\cO_X[n])\\ \xrightarrow{-\circ f} & \Hom(\cO_p,\cO_X[n])\to \dots.\notag
    \end{align}
Note that $\Hom(\cI_q[n],\cO_X[n])=\C$. Let $g$ be a non-zero element in $\Hom(\cI_q[n],\cO_X[n])$, then $\Cone(g)\cong \cO_q[n]$. Apply $\Hom(\cO_p,-)$ to the distinguished triangle
\begin{align}
        \cI_q[n]\xrightarrow{g}\cO_X[n]\to \cO_q[n]\xrightarrow{+}, \label{eq:2.100}
    \end{align}

As $\RHom(\cO_p,\cO_q)=0$, the map $g\circ -:\Hom(\cO_p,\cI_q[n])\to \Hom(\cO_p,\cO_X[n])$ is an isomorphism. It follows that the composition $g\circ f\neq 0$.

Therefore, in \eqref{eq:2.111}, the map $\Hom(\cI_q[n],\cO_X[n])\xrightarrow{-\circ f}\Hom(\cO_p,\cO_X[n])$ is non-zero. Note that  $\hom(\cI_q[n],\cO_X[n])=\hom(\cO_p,\cO_X[n])=1$, the map $-\circ f$ is an isomorphism. It follows that $\Hom(F,\cO_X[n])=0$. By Serre duality, $\Hom(\cO_X,F)=0$.

Together with \eqref{eq2112}, we get $\RHom(\cO_X,F)=0$, in other words, the object $F\in\cO_X^\perp$. 
 \end{proof}  

The similar construction as that in Example \ref{eg:spherical} works for all $n$-spherical objects in $\Db(X)$. The statement reads as follows.

\begin{Prop}\label{prop:sphericalorth}
Let $X$ be an $n$-dimensional smooth proper variety ($n\geq 1$) over $\C$. Then for every $n$-spherical object $E$ in $\Db(X)$, its orthogonal category $E^\perp$ has non-trivial objects.    
\end{Prop}
\begin{proof}
     By semicontinuity, we may pick two different closed points $p$ and $q$ on $X$ (denote by $P=\cO_p$ and $Q=\cO_q$ the skyscraper sheaves) such that $\RHom(E,P)\cong\RHom(E,Q)$.

      By Serre duality, we may consider the object $$F:=E\otimes \RHom(E,P)\cong E\otimes (\RHom(Q[-n],E))^*.$$ 
    
    We complete the commutative square 
     \begin{center}
	\begin{tikzcd}
		 Q[-n] \arrow{r} \arrow{d}{\ev_1} & 0\arrow{d} \\
		 F
		 \arrow{r}{\ev_2} & P 
	\end{tikzcd}
\end{center}

    to the following diagram of distinguished triangles (arrows `$\xrightarrow{+}$' are all  omitted).

    \begin{center}
	\begin{tikzcd}
		Q[-n] \arrow{d} \arrow{r}{\id}
		& Q[-n] \arrow{r} \arrow{d}{\ev_1} & 0\arrow{d} \\
		S_P \arrow{r} \arrow{d}& F
		\arrow{d} \arrow{r}{\ev_2} & P \arrow{d}{\id}\\
		K \arrow{r} & S'_Q\arrow{r}{f} & P.
	\end{tikzcd}
\end{center}

All squares except the bottom-left one are commutative.

    Apply $\Hom(-,E)$ to the distinguished triangle $Q[-n]\xrightarrow{\ev_1}F\rightarrow S'_Q\xrightarrow{+}$. We get surjective maps $\Hom(F,E[i])\xrightarrow{-\circ\ev_1}\Hom(Q[-n],E[i])$ for every $i\in\Z$ due to the nature of the evaluation map $\ev_1$. So in the long exact sequence
\begin{align*}
   \dots &\xrightarrow{-\circ\ev_1}\Hom(Q[-n],E[i-1])\\
     \xrightarrow{g_{i-1}}\Hom(S'_Q,E[i])\rightarrow \Hom(F,E[i])&\xrightarrow{-\circ\ev_1}\Hom(Q[-n],E[i])\\
     \xrightarrow{g_i} \Hom(S'_Q,E[i+1])\rightarrow \dots, &
\end{align*}
the map $g_i=0$ for every $i\in\Z$.
    
    It follows that

    \begin{equation}
    \hom(F,E[i])=\hom(Q[-n],E[i])+\hom(S'_Q,E[i])\label{eq1}
    \end{equation}
for every $i\in\Z$.

    As $E$ is $n$-spherical, by the definition of $F=E\otimes \RHom(E,P)$, we have \begin{equation}
        \hom(F,E[i])=\hom(E,P[-i])+\hom(E,P[n-i])\label{eq2}
    \end{equation} for every $i\in\Z$.
    
    By Serre duality, $\hom(Q[-n],E[i])=\hom(E,P[-i])$ for every $i\in\Z$. It follows by \eqref{eq1}, \eqref{eq2}, and Serre duality  that 
    \begin{equation}
        \hom(E,P[n-i])=\hom(S'_Q,E[i])=\hom(E,S'_Q[n-i])\label{1}
    \end{equation}
    for every $i\in\Z$.
    
    Finally, we apply $\Hom(E,-)$ to the commutative diagram 
     \begin{center}
	\begin{tikzcd}
		 F
		\arrow{d} \arrow{r}{\ev_2} & P \arrow{d}{\id}\\
		 S'_Q\arrow{r}{f} & P.
	\end{tikzcd}
\end{center}
 By definition of the evaluation map $\ev_2$, the map $\Hom(E,F[i])\xrightarrow{\ev_2\circ -}\Hom(E,P[i])$ is surjective for every $i\in\Z$. It follows that the map $\Hom(E,S'_Q[i])\xrightarrow{f\circ -}\Hom(E,P[i])$ is surjective for every $i\in \Z$ as well. By \eqref{1}, the map $\Hom(E,S'_Q[i])\xrightarrow{f\circ -}\Hom(E,P[i])$ is isomorphic for every $i\in \Z$. Apply $\Hom(E,-)$ to the distinguished triangle $$K\to S'_Q\xrightarrow{f} P\xrightarrow[]{+}$$ at the bottom. It follows that $\Hom(E,K[i])=0$ for every $i\in\Z$.

 Therefore, the object  $ K\in E^\perp$ and it is clearly non-zero. \end{proof}

\begin{Def}\label{def:sphericaltwist}
Let $E$ be a shperical object in $\Db(X)$, the \textit{spherical twist} $\mathsf{T}_E$ is defined to be the Fourier--Mukai transformation from $\Db(X)$ to $\Db(X)$ with kernel $\Cone(E^\vee\boxtimes E\xrightarrow{\mathsf{can}}\cO_{\Delta})$.
\end{Def}
 The spherical twist $\mathsf{T}_E$ is an exact autoequivalence on $\Db(X)$ (see \cite[Theorem 1.2]{Seidel-Thomas:braid}). In practice, the spherical twist  $\mathsf{T}_E$ on an object $F$  can  be computed by the equivalent definition as that in \cite[Definition 2.5]{Seidel-Thomas:braid}: $\mathsf{T}_E(F)=\Cone(E\otimes\RHom(E,F)\xrightarrow{\ev} F)$.

It is clear that $\mathsf{T}_E(E)=E[1-n]$ and $\mathsf{T}_E$ fixes objects in $E^\perp$. Proposition \ref{prop:sphericalorth} gives the following corollary immediately.
\begin{Cor}\label{cor:sphericaltwistnonzero}
    Let $X$ be an $n$-dimensional proper smooth variety over $\C$ and $E$ be an $n$-spherical object, then $\mathsf{T}_E^k$ is non-trivial unless $k=0$.
\end{Cor}

By the philosophy on the relation between even-dimensional nodal singularity and $2$-spherical object as that in \cite{shengxuannodal} and \cite{kuznetsov2023categorical}, one may reach the Question \ref{ques:nodalsg} in the introduction. Namely, let $S$ be a smooth projective surface with a $2$-spherical object $E$ in $\Db(S)$.  Do we always have the quotient category\begin{align*}
        \Db(S)/\langle E^\perp,E\rangle\cong \mathrm{D}_{\mathrm{sing}}(\Spec(\C[x]/(x^2)))?
    \end{align*}

Here we write $\mathrm{D}_{\mathrm{sing}}(X)=\Db(X)/\mathrm{D}_{\mathrm{perf}}(X)$ for the singular category of a variety $X$.

\begin{Rem}\label{rem:sphericalperp}We finish this section with some remarks on Question \ref{ques:nodalsg}.
    \begin{enumerate}
    \item For a special case and the starting point of this question, assume that the surface $S$ contains a smooth integral rational curve $C\cong \mathbf P^1$ with self-intersection $-2$. Then $E=\cO_C(-1)$ is a $2$-spherical object. Denote by $\pi:S\to T$ the morphism contracting $C$. Then $\pi_*$ induces the equivalence $\Db(S)/\langle \cO_C(-1)\rangle\cong \Db(T)$ and $\pi^*(\mathrm{D}^{\mathrm{perf}}(T))=\langle E^\perp\rangle$. It follows that 
$$\Db(S)/\langle E^\perp,E\rangle\cong \Db(T)/\mathrm{D}^\mathrm{perf}(T)=\mathrm{D}_{\mathrm{sing}}(T)\cong \mathrm{D}_{\mathrm{sing}}(\Spec(\C[x]/(x^2))).$$
        \item The even-dimensional nodal singularity category $\mathrm{D}_{\mathrm{sing}}(\Spec(\C[x]/(x^2)))$  is generated by the unique simple object $A$ with $A=A[1]$. Each object in $\mathrm{D}_{\mathrm{sing}}(\Spec(\C[x]/(x^2)))$ is of the form  $A^{\oplus m}$ for some $m\in\Z_{\geq 0}$. By the argument in Proposition \ref{prop:sphericalorth}, one may check that $\underline{\cO_p}=\underline{\cO_q}=\underline{\cO_q}[1]$ in $\Db(S)/\langle E^\perp,E\rangle$ for any points $p$ and $q$ on the support of $E$. One may further show that every object in $\Db(S)/\langle E^\perp, E\rangle$ is equivalent to $\cO_p$. This gives a piece of evidence for the question.

\item The question may also have a more ambitious version in the general case. Replace $S$  by an arbitrary smooth proper variety $X$ and $E$  by a Serre invariant full subcategory $\cA$ in $\Db(X)$. One may ask whether $\Db(X)/\langle \cA^\perp,\cA\rangle\cong \cA^\dag$ which is independent on the choice of $X$. 
 \end{enumerate}
\end{Rem}
\section{Spherical bundles and stability}\label{sec3}
From this section, we will focus on smooth projective surfaces rather than varieties with arbitrary dimensions. We briefly recap some standard notions from \cite{HL:Moduli} on the slope and Gieseker stability of coherent sheaves on a surface.  As we will be mainly interested in torsion-free coherent sheaves, some notions are simplified in the note.

\subsection{Recap: Stability, Mukai vector}
\begin{Def}\label{def:stab}

Let $S$ be a smooth projective surface and $H$ be an ample divisor. For a torsion-free sheaf $E$, the \textit{reduced Hilbert polynomial} of $E$ is defined to be $p_H(E,m):=\frac{\chi(E(mH))}{\rk(E)}$. The \textit{slope} of $E$ is defined to be $\mu_H(E):=\frac{H.\ch_1(E)}{\rk(E)}$.

A torsion-free coherent sheaf $E$ is called Gieseker (semi)stable with respect to $H$ (or $p_H$-(semi)stable) if $p_H(F, m) < (\leq) p_H(E, m)$ when $m\gg 0$, for all proper non-trivial subsheaves $F \subset E$. The sheaf $E$ is called slope (semi)stable with respect to $H$ (or $\mu_H$-(semi)stable) if $\mu_H(F) <  (\leq) \mu_H(E)$ for all non-trivial subsheaves $F \subset E$ with $\rk(F)<\rk(E)$.
    
\end{Def}

For a given polarization $H$, the following slogan is well-known:
\begin{align}\label{eq:slopestableslogan}
    \text{slope stable}\implies \text{Gieseker stable}\implies\text{Gieseker semistable}\implies\text{slope semistable.}
\end{align} 

When $E$ is strictly  Gieseker semistable, in other words, semistable but not stable, $E$ admits a Jordan--H\"older filtration $0=E_0\subset E_1\subset\dots\subset E_n=E $ such that all quotients $E_{i+1}/E_i$ are stable with the same reduced Hilbert polynomial as that of $E$. The direct sum of all factors $\oplus_{i=1}^n(E_i/E_{i-1})$ does not depend on the filtration. When $E$ is strictly slope semistable, a Jordan--H\"older filtration exists such that all quotients $E_{i+1}/E_i$ are slope stable with the same slope as that of $E$. The direct sum of all factors $\oplus_{i=1}^n(E_i/E_{i-1})$ is only unique up to a modification of skyscraper sheaves.

\begin{Def}\label{def:mukaivector}
Let $S$ be a smooth projective K3 surface. The \emph{Mukai vector} of an object $E\in\Db(S)$ is defined to be:
\begin{align*}
    v(E):=(\rk(E),\ch_1(E),\ch_2(E)+\rk(E))\in \tilde{\mathrm{H}}(S,\Z).
\end{align*}
The lattice product structure on $\tilde{\mathrm{H}}(S,\Z)=\mathrm H^0(S,\Z)\oplus \NS(S)\oplus \mathrm H^4(S,\Z)$ is defined to be \begin{align*}
    \langle (r_1,\ell_1,c_1),(r_2,\ell_2,c_2)\rangle:=\ell_1.\ell_2-r_1c_2-r_2c_1.
\end{align*}
We call a Mukai vector $v=(\rk,\ell,c)\in\tilde H(S,\Z)$ \textit{spherical} if  $\rk>0$ and $\langle v,v\rangle =-2$. 
\end{Def}
Here we add the condition $\rk>0$ in the definition to simplify some statements in the future. It is clear by Riemann--Roch that $\chi(E,F)=-\langle v(E),v(F)\rangle$. The Mukai vector $v(E)$ of a spherical bundle $E$ is spherical.

\subsection{Unstable spherical bundles}
Before giving an example of a spherical bundle on a K3 surface that is never stable, it is worth mentioning that a spherical coherent sheaf can be non-torsion-free but with a positive rank. 

\begin{Ex}\label{eg:torgluespherical}
Let $S$ be a K3 surface containing a smooth rational curve $C$. Then we may consider the object $\mathsf{T}_{\cO_S}(\cO_C(-2))$ which is given by the distinguished triangle:
\begin{equation}\label{eq3.1}
    \to \cO_S[-1]\xrightarrow{\ev}\cO_C(-2) \to \mathsf{T}_{\cO_S}(\cO_C(-2)) \to \cO_S\to.
\end{equation}
As $\cO_C(-2)$ is spherical and $\mathsf{T}_{\cO_S}$ is an autoequivalence, the object $\mathsf{T}_{\cO_S}(\cO_C(-2))$ is spherical. As $\mathsf{T}_{\cO_S}(\cO_C(-2))$ is the extension of two coherent sheaves $\cO_C(-2)$ and $\cO_S$, it is a coherent sheaf. It is clear that $\mathsf{T}_{\cO_S}(\cO_C(-2))$ is not  torsion-free.
\end{Ex}

We now make a further spherical twist on $\mathsf{T}_{\cO_S}(\cO_C(-2))$ to get an example of a spherical bundle that is never stable. 
\begin{Ex}\label{eg:nonstablesphericalbundle}
    Let $S$ be a smooth projective K3 surface with Picard number $2$ containing two smooth integral rational curves (the existence of such a K3 surface is ensured by  \cite[Theorem 2]{Kovacs:K3}). In particular, the cone $\Nef(S)$ is strictly contained in $\Peff(S)$.

Denote one of the smooth rational curves by $C$ and its divisor class by $[C]$. Let $\beta$ be a non-zero divisor in $[C]^\perp\cap \partial\Nef(S)$. Then the two rays of $\partial \Nef(S)$ are spanned by $\beta$ and $\beta-a_1[C]$ for some $a_1>0$. 

 The two rays of $\partial\Peff(S)$ are spanned by $[C]$ and $\beta-a_2[C]$ for some $a_2>a_1>0$. Moreover, it follows from the duality between $\Nef(S)$ and $\Peff(S)$ that 
\begin{align}\label{eq4}
     \beta^2= 2a_1a_2.
\end{align}

As $a_2>a_1$, there exists a very ample divisor $[G]=s(\beta-a[C])$ for some $a\in\Q$ such that $$a_2-\sqrt{a_2^2-a_1a_2}<a<a_1.$$

Claim: For every integer $m\gg 0$, $F_m:=\mathsf{T}_{\cO_S(-mG)}(\mathsf{T}_{\cO_S}(\cO_C(-2)))[-1]$ is a spherical bundle that is not slope semistable with respect to any polarization $H$.
\end{Ex}
\begin{proof}
Apply $\mathsf T_{\cO_S(-mG)}$ to the short exact sequence:
$$0\to \cO_C(-2)\to \mathsf T_{\cO_S}(\cO_C(-2))\to \cO_S\to 0.$$

We have the following diagram of distinguished triangles (arrows `$\xrightarrow[]{+}$' are omitted for simplicity):
 \begin{center}
	\begin{tikzcd}
		A_m \arrow{d} \arrow{r}{}
		& F_m \arrow{r} \arrow{d}{} & B_m\arrow{d} \\
		\cO_S(-mG)^{\oplus mGC-1} \arrow{r} \arrow{d}{\ev_1}& \cO_S(-mG)^{\oplus mGC+m^2G^2/2+1}
		\arrow{d}{\ev} \arrow{r}{} & \cO_S(-mG)^{\oplus m^2G^2/2+2} \arrow{d}{\ev_2}\\
		\cO_C(-2) \arrow{r} & \mathsf{T}_{\cO_S}(\cO_C(-2))\arrow{r} & \cO_S,
	\end{tikzcd}
 \end{center}
where $A_m=\mathsf T_{\cO_S(-mG)}(\cO_C(-2))[-1]$ and $B_m=\mathsf T_{\cO_S(-mG)}(\cO_S)[-1]$. Note that $\ev_1$ and $\ev_2$ are both surjective in $\Coh(S)$ when $m\gg 0$, it follows that the objects $A_m$ and $B_m$ are torsion-free spherical sheaves, hence spherical bundles. So $F_m$ is a spherical bundle. \\

 To finish the claim, we show that when $m\gg 0$, the slope $\mu_H(A_m)>\mu_H(B_m)$ for every polarization $H$ of $S$.

 It follows by a direct computation that:
 \begin{align*}
     \mu(A_m)-\mu(B_m) & = \frac{-mG(mGC-1)-C}{mGC-1} - \frac{-mG(\frac{m^2G^2}{2}+2)}{\frac{m^2G^2}{2}+1} \\
     & = \frac{2mG}{m^2G^2+2}-\frac{C}{mGC-1}\\
     & = \frac{2ms}{m^2G^2+2}\left(\beta-aC-\frac{m^2G^2+2}{2ms(mGC-1)}C\right).
 \end{align*}

Note that 
\begin{align*}
    \lim_{m\to +\infty} a+ \frac{m^2G^2+2}{2ms(mGC-1)} & = a+\frac{G^2}{2sGC} =a+\frac{s^2\beta^2-2a^2s^2}{4s^2a} \\ 
   & =\frac{a_1a_2+a^2}{2a}<a_2.
\end{align*}

Therefore, when $m\gg 0$, the divisor $ \mu(A_m)-\mu(B_m)$ is effective. For every ample divisor $H$, it follows that $\mu_H(A_m)-\mu_H(B_m)=H(\mu(A_m)-\mu(B_m))>0$. By definition, the spherical bundle $F_m$ is not $\mu_H$-semistable. \end{proof}

\begin{Rem}\label{rem:unstablesph}
    We do not know if there are such kind of never-stable spherical bundles on a K3 surface with other types of Mori cones, for example, an elliptic K3 surface admitting a section.

    For a given spherical Mukai vector, we also do not know if there can be infinitely many never-stable spherical bundles with that vector
    \end{Rem}
\begin{Ques}\label{ques:sphbristab}
    A related question is that: given a spherical object (or semi-rigid object, see \cite[Definition 2.4]{K3Pic1}) $E\in\Db(S)$, does there exist a Bridgeland stability condition $\sigma$ such that $E$ is $\sigma$-semistable?
\end{Ques}
 Note that when the K3 surface $S$ is of Picard number one, the question above is true due to \cite[Lemma 6.3]{K3Pic1} (or in the spherical case, more precisely stated in Corollary 6.9).
\subsection{Infinitely many spherical bundles}

In this section, we give examples of infinitely many stable spherical bundles with the same Mukai vector. In general, we expect there exist such examples when the cone $\Nef(S)$ is circular. 

\begin{Ex}\label{eg:infinitelymanyspherical}
Let $\Lambda$ be a rank $3$ primitive sublattice in $ U^{\oplus 3}\subset U^{\oplus 3}\oplus (-E_8)^{\oplus 2}$ with an orthogonal basis $\{\mathbf h, \mathbf e,\mathbf f\}$ satisfying: $\mathbf h^2=12$, $\mathbf e^2=-6$ and $\mathbf f^2=-30$.  By \cite[Corollary 1.9]{Morrison:K3}, there exists a projective  K3 surface $S$  such that $\NS(S)\cong \Lambda$.

We claim that there exist infinitely many spherical bundles with Mukai vector $(2,\mathbf f,-7)$ on $S$. Each of them is slope stable with respect to some polarization on $S$.
    \end{Ex}

\begin{proof}
  
For every pair of  positive integer solutions $(a,b)$  to the Pell equation $2x^2-y^2=1$, we may consider the divisors 
$D_a=a\mathbf h+b\mathbf e+\mathbf f$ and $E_a=-a\mathbf h-b\mathbf e$. 

As $\NS(S)$ has no class with self-intersection $-2$, its effective cone is circular and is equal to the nef cone. A divisor $D=x\mathbf h+y\mathbf e+z\mathbf f\in \Peff(S)=\Nef(S)$ if and only if $D^2\geq 0$ and $x\geq 0$.  

Note that $$(D_a-E_a)^2=48a^2-24b^2-30=-6,$$ neither the divisor $D_a-E_a$ or $E_a-D_a$ is effective. It follows that $\Hom(\cO_S(D_a),\cO_S(E_a))=\Hom(\cO_S(E_a),\cO_S(D_a))=0$

By Riemann--Roch, $\chi(\cO_S(E_a),\cO_S(D_a))=\tfrac{1}{2}(D_a-E_a)^2+2=-1$. So $\Ext^1(\cO_S(E_a),\cO_S(D_a))=\C$, and  there is a unique rank $2$ bundle $V_{a}$, non-trivially extended by $\cO_S(D_a)$ and $\cO_S(E_a)$. Namely, it fits into the short exact sequence
$$0\to \cO_S(D_a)\to V_{a}\to \cO_S(E_a)\to 0.$$

The bundle $V_a$ is also given as $\mathsf T_{\cO_S(E_a)}(\cO_S(D_a))$ by definition. So it is spherical. Its Mukai vector is given as

$$v(V_a)=(1,D_a,\tfrac{1}{2}D_a^2+1)+(1,E_a,\tfrac{1}{2}E_a^2+1)=(2,\mathbf f,-7).$$

For two different pairs of positive integer solutions $(a_1,b_1)$ and $(a_2,b_2)$ with $a_1>a_2$. By the criterion of the effective divisor, $D_{a_2}-D_{a_1}$ and $E_{a_2}-D_{a_1}$ are not effective.
It follows that $\Hom(\cO_S(D_{a_1}),V_{a_2})=0$. 

Therefore $V_{a_1}$ and $V_{a_2}$ are not isomorphic to each other when $a_1\neq a_2$. As the Pell equation $2x^2-y^2=1$ has infinitely many positive integer solutions, there are infinitely many spherical bundles $V_{a}$ with the same Mukai vector $(2,\mathbf f,-7)$.\\

We then show the stability of $V_a$'s. For every $V_{a}$, we may consider the divisor $H_a:=7a\mathbf h+7b\mathbf e+3\mathbf f$. By a direct computation, $H_a^2=588a^2-294b^2-270=24>0$ and $7a>0$.  So the divisor $H_a$ is ample.

By a direct computation, $$H_aD_a=42-90=-48,\; H_aE_a=-42,\;\text{ and }\mu_{H_a}(V_a)=-45.$$ Suppose $V_{a}$ is not $\mu_{H_a}$-stable, then it must be destabilized by a line bundle. So there exists a sub line bundle $\cO_S(C)$ with $H_aC=\mu_{H_a}(C)\geq \mu_{H_a}(V_{a})=-45$.

As $\cO_S(C)\subset V_a$, we have  $\Hom(\cO_S(C),V_a)\neq 0$. If $\Hom(\cO_S(C),\cO_S(D_a))\neq 0$, then $D_a-C$ is effective. It follows that $H_aC\leq H_aD_a=-48$, which contradicts the conclusion that $H_aC\geq -45$. Hence, $\Hom(\cO_S(C),\cO_S(E_a))\neq 0$. As the extension is non-trivial, $C\neq E_a$.  So $E_a-C$ is non-zero effective.  It follows that $H_aC=H_aE_a-H_a(E_a-C)\leq -42-6<-45$, which is again a contradiction. Therefore, $V_{a}$ cannot be destabilized by any line bundle and it is $\mu_{H_a}$-stable.
  \end{proof}

\section{Walls and chambers}\label{sec4}
\subsection{A finite theory for walls and chambers}
Let $S$ be a smooth projective surface of Picard number at least $2$. Let $v=(\rk,\ell,s)\in \widetilde{\mathrm{H}}(S,\Z)$ satisfying $\rk>0$ and $(\rk,\ell)$ being primitive, in other words, $(\rk,\ell)\neq m(\rk',\ell')$ for any $m\in \Z_{\geq2}$. Rescaling the polarization does not affect the stability of any sheaves, so for any ample divisor $H$ and element $H'\in \R_{>0}. H\subset\Amp(S)$, we may define the (semi)stability with respect to $H'$ and $\mu_{H'}$-slope (semi)stability the same as that with respect to $H$. Denote by $\Amp_{\Q}(S)$ the subset of all rational points in $\Amp(S)$. For every $H\in \Amp_{\Q}(S)$, denote by $M^s_H(v)$ ($M^{ss}_H(v)$) the moduli space of $\mu_H$-slope (semi)stable torsion-free sheaves with class $v$.

In this section, we review the wall and chamber structure for $M^s_H(v)$  while the polarization $H$ is changing in $\Amp(S)$. The upshot is a finite theory on the numerical walls when the nef cone is rational polyhedral. For simplicity, we focus on slope stability and give a remark on the theory of Gieseker stability at the end of this section. 

For a coherent sheaf $F$ on $S$, we will denote by $\ch(F):=(\rk(F),\ch_1(F),\ch_2(F))$.  Note that by a little abuse of notations, in this section, the vector $v\in \widetilde{\mathrm{H}}(S,\Z)$ is assumed to be $\ch(F)$ which is different from the Mukai vector $v(F):=(\rk(F),\ch_1(F),\ch_2(F)+\rk(F))$ in the K3 surface case. The discriminant is denoted as $\Delta(v):=(\ch_1)^2-2\rk\ch_2$.

%More details are referred to \cite{HL:Moduli,Qin:chamber}. 
\begin{Lem}[{\cite[Lemma 4.C.5]{HL:Moduli}}]\label{lem:chaningpolarizationsegment}
    Let $F$ be a torsion-free sheaf that is $\mu_H$-slope stable with respect to some $H\in\Amp_\Q(S)$. Let $H'\in\Amp_\Q(S)$. 
    \begin{enumerate}[(1)]
        \item If $F$ is $\mu_{H'}$-slope stable, then $F$ is $\mu_{L}$-slope stable for every $L=(1-t)H+tH'$, where $t\in [0,1]\cap \Q$.
        \item If $F$ is not $\mu_{H'}$-slope stable, then $\exists L=(1-t)H+tH'$ for some $t\in (0,1]\cap \Q$ such that $F$ is strictly $\mu_L$-slope semistable.\hfill $\Box$
    \end{enumerate}
\end{Lem}

For a pair of classes $v=(\rk,\ell,s)$ and $v'=(\rk',\ell',s')\in\tilde{\mathrm{H}}(S,\Z)$. We denote by $\cW(v,v'):=(\rk\cdot \ell'-\rk'\cdot \ell)^\perp$ whenever this is codimension-$1$ $\R$-linear subspace  in $\NS(S)_\R$.

\begin{Def}\label{def:numwall}
Let $v=(\rk,\ell,s)\in\tilde{\mathrm{H}}(S,\Z)$ with $\rk>0$ and $(\rk,\ell)$ being primitive.   We call a codimension-$1$ $\R$-linear subspace $\cW$ in $\NS(S)_{\R}$ a \textit{numerical wall} of $v$ if $\cW=\cW(v,v_i)$ for some $ v_i=(\rk_i,\ell_i,s_i)\in \widetilde{\mathrm{H}}(S,\Z)$, $i=1,2$, such that:
  \begin{enumerate}
      \item $v=v_1+v_2$, $\rk_i>0$, and $\Delta(v_i)\geq 0$;
      \item $\cW\cap \Amp(S)\neq \emptyset$.
  \end{enumerate}
\end{Def}

\begin{Lem}\label{lem:ssemiimpliesonwall}
      Let $v=(\rk,\ell,s)\in\tilde{\mathrm{H}}(S,\Z)$ with $\rk>0$ and $(\rk,\ell)$ being primitive. Let $F$ be a strictly $\mu_H$-slope semistable torsion-free sheave with class $v$. Then $H$ is on a numerical wall of $v$.
\end{Lem}
\begin{proof}
   Let $F_1\subset F$ be a proper saturated subsheaf of $F$ with $\mu_H(F_1)=\mu_H(F)$. Let $v_1=\ch(F_1)$ and $v_2=\ch(F_2)$, then the only non-trivial part $\Delta(v_i)\geq 0 $ follows from the Bogomolov inequality, see \cite[Theorem 3.4.1]{HL:Moduli}.
\end{proof}

\begin{Prop}\label{prop:finwalls}
    Let $S$ be a smooth projective surface over $\C$ with $\Nef(S)$ being rational polyhedral. Then for every class $v=(\rk,\ell,s)\in\tilde{\mathrm{H}}(S,\Z)$ with $\rk>0$ and $(\rk,\ell)$ being primitive, there are only finitely many numerical walls of $v$. In particular, there are finitely many walls and chambers for the moduli space $M^s_H(v)$ in $\Amp(S)$.
\end{Prop}

We will reduce the argument for the proposition to some elementary lattice theory.

Let $N=\R^n$ be an $n$-dimensional real vector space and $\Lambda=\Z^n$ the integer points in it.  Let $|\bullet|$ be a standard norm on $N$. For a vector $\cud\in N$ and $r\in\R_{>0}$, denote by $B(\cud,r):=\{\cuc\in N\;|\; |\cud-\cuc|<r\}$ the neighbourhood at $\cud$. Denote by $\tilde{B}(\cud,r):=\{\R_{>0}.\cuc\;|\; \cuc\in B(\cud,r) \}$ the cone neighbourhood. 

Let $Q\in \mathrm{Mat}(n\times n,\Z)$ be a non-degenerate symmetric bilinear form on $N$ with $\Z$-coefficients  and signature $(1,n-1)$. Denote by $C_0=\{x\in N\;|\;Q(x)>0\}=C_0^+\coprod C_0^-,$ where $C_0^\pm$ are the two connected components of $C$. It follows by H\"older's Inequality that $Q(\cud,\cud')>0$ for all $\cud,\cud'\in C^+_0$. For a vector $\cud\in N$, we denote by $\cud^\perp:=\{\cuc\in N\;|\; Q(\cud,\cuc)=0\}$.
\begin{Lem}\label{lem:toylattice}
Let $A$ be a non-empty open convex cone in $C_0^+$ such that 
\begin{align}\label{eq4.1}
    \partial A\cap\partial C_0^+ = \cup_{i=1}^k \R_{>0}.\cud_i
\end{align}
for finitely many $\cud_i$ in $\Lambda$ such that $ \overline{A}$ is locally polyhedral at every $\cud_i$.

Then for any real value $t<0$,
    $$\#\{\cud\in \Lambda\;|\; Q(\cud,\cud)\geq t,\exists \;\cuh\in A, Q(\cud,\cuh)=0\}<+\infty.$$
\end{Lem}
\begin{proof}
   Denote by $A^*:=\{\cud\in N\;|\; Q(\cud,\cuh)\geq 0 $ for all $\cuh\in A\}$ the closed dual cone of $A$. 

    It is clear that $A^*\supset C^+_0$, and  %by H\"older's Inequality and 
    
    \begin{align}\label{eq442}
     & \partial A^*\cap \partial C^+_0=\partial A\cap \partial C^+_0=\cup_{i=1}^k \R_{>0}.\cud_i;\\
      &  \{\cud\in \Lambda\;|\; \exists\; \cuh\in A, Q(\cud,\cuh)=0\}=\{0\}\cup (\Lambda\setminus\mathrm{Int}(A^*\cup(-A^*)).\label{eq443}
    \end{align}

    For each $\cud_i$ spanning a ray of $\partial A\cap \partial C^+_0$, as $A$ is locally polyhedral and $C^+_0$ is circular, the boundary of the dual cone $A^*$ is given by $\cud_i^\perp$ locally at $\cud_i$. In other words, $\exists r_i>0$ such that 
    \begin{align}\label{eq4.3}
     &\partial A^*\cap B(\cud_i,r_i)=\cud_i^\perp\cap B(\cud_i,r_i); \\
     & \forall \cud\in B(\cud_i,r_i)\setminus A^*\text{, we have } Q(\cud,\cud_i)<0\text{ and  } \cud -\tfrac{1}2\cud_i\notin A^*\cup(-A^*).\label{eq4.4}
    \end{align}

    Since $\cud_i\in\Lambda$ and $Q(-,-)$ is with coefficient in $\Z$, the sublattice $\cud_i^\perp\cap \Lambda$ is of rank $n-1$ and $\exists c_i\in \Q_{>0}$ such that for every $\cud\in\Lambda$, the distance from $\cud$ to $\cud_i^\perp$ is $nc_i$ for some $n\in\Z$.

    Let $\epsilon_i:=\min\{r_i,-\frac{c_i}{t}\}$, then for every $\cuc\in(\tilde B(\cud_i,\epsilon_i)\cap \Lambda)\setminus A^*$, by \eqref{eq4.3}, the distance from $\cuc$ to $\R.\cud_i$ is at least $c_i$. It follows that $\cuc\in B(m\cud_i,m\epsilon_i)$ for some $m\epsilon_i\geq c_i$.

    By \eqref{eq4.4}, the vector $\cuc-\frac{m}2\cud_i\notin A^*\cup(-A^*)$, so $ Q(\cuc-\frac{m}2\cud_i,\cuc-\frac{m}2\cud_i)<0$. As $Q(-,-)$ is with $\Z$-coefficient and both $\cud_i,\cuc\in \Lambda$, by \eqref{eq4.4}, $Q(\cuc,\cud_i)\leq -1$.  It follows that \begin{align}\label{eq4.5}
        Q(\cuc,\cuc)& = Q(\cuc-\tfrac{m}2\cud_i,\cuc-\tfrac{m}2\cud_i)+mQ(\cuc,\cud_i)-\tfrac{m^2}4Q(\cud_i,\cud_i)<-m\leq -\tfrac{c_i}{\epsilon_i}\leq t.
    \end{align}
    So
    \begin{align}\label{eq4.7}
        \forall\; \cuc \in (\pm \tilde{B}(\cud_i,\epsilon_i)\cap \Lambda)\setminus(\pm A^*)=\pm((\tilde B(\cud_i,\epsilon_i)\cap \Lambda)\setminus A^*)\text{, we have }Q(\cuc,\cuc)<t.
    \end{align}

    Let  $M:=\partial B(\mathbf{0},1)\setminus\left(\pm\left(\mathrm{Int} A^*\bigcup \left(\cup_{i=1}^k\tilde B_i(\cud_i,\epsilon_i)\right)\right)\right)$. By \eqref{eq442}, $\mathrm{Int} A^*\bigcup \left(\cup_{i=1}^k\tilde B_i(\cud_i,\epsilon_i)\right)\cup\{\mathbf 0\}\supset \overline{C^+_0}$, so $M\cap \overline{C}_0=\emptyset$. Since $M$ is compact, the function $Q(\cud):=Q(\cud,\cud)$ has minimum value on $M$. As $\overline C_0=\{\cud\in N\;|\; Q(\cud,\cud)\geq 0\}$,  the value $\min_{\cud\in M}\{Q(\cud,\cud)\}=-\delta<0.$
    
It follows that 
\begin{align}\label{eq4.8}
    \forall\; \cud\notin \mathrm{Int} A^*\bigcup \left(\cup_{i=1}^k\tilde B_i(\cud_i,\epsilon_i)\right)\text{ and } |\cud|>\sqrt{-\tfrac{t}\delta}\text{, we have } Q(\cud,\cud)<t.
\end{align}
    As summary, by (\ref{eq443}, \ref{eq4.7}, \ref{eq4.8}),
    \begin{align*}
        \{\cud\in \Lambda\;|\; Q(\cud,\cud)\geq t,\exists \;\cuh\in A, Q(\cud,\cuh)=0\} \subset \Lambda\cap B\left(\mathbf{0}, \sqrt{-\tfrac{t}{\delta}}\right),
    \end{align*}
    which is finite.
\end{proof}

Now we can prove the main result on finite walls.

\begin{proof}[Proof of Proposition \ref{prop:finwalls}]
    Assume $v=v_1+v_2$ for some $v_i=(r_i,\ell_i,s_i)$ satisfying conditions in Definition \ref{def:numwall}. Denote by $\delta_i=r\cdot\ell_i-r_i\cdot \ell$, then by Definition \ref{def:numwall}.(b), $\exists H\in \Amp(S)$ such that $\delta_i.H=0$. By Hodge Index Theorem, $\delta_i^2\leq 0$.
    
    It follows by Definition \ref{def:numwall}.(a) that\begin{align*}
        &\Delta(v_i)=\ell_i^2-2r_i s_i\geq 0\\
        \implies & \delta_i^2 +2rr_i\ell.\ell_i-r^2_i\ell^2-2r_ir^2s_i\geq 0
        \iff  \tfrac{r}{r_i}\delta_i^2+2r^2\ell.\ell_i-r_ir\ell^2-2r^3s_i\geq 0\\
        \implies & \delta_i^2+2r^2\ell.\ell_i-r_ir\ell^2-2r^3s_i\geq \tfrac{r_i-r}{r_i}\delta^2_i\geq 0\\
         \implies & \delta_1^2+\delta_2^2+r^2\ell^2-2r^3s = \sum_{i=1,2}(\delta_i^2+2r^2\ell.\ell_i-r_ir\ell^2-2r^3s_i)\geq 0\\
         \implies & \delta_1^2\geq -r^2\ell^2+2r^3s-\delta^2_2\geq -r^2\ell^2+2r^3s.
   \end{align*}
   Let $NS(S)_\R$ be $N$, $NS(S)$ be $\Lambda$ equipped with the intersection as the bilinear form. Let $\Nef(S)$ be $A$, as it is rational polyhedral, the assumption \eqref{eq4.1} is satisfied. If $-r^2\ell^2+2r^3s\geq 0$, then it can be replaced by $-1$. So we may assume $-r^2\ell^2+2r^3s< 0$  By Lemma \ref{lem:toylattice}, 
   $$\#\{\delta_1\in\NS(S)\;|\;\delta^2_1\geq -r^2\ell^2+r^3s,\exists\; H\in\Amp(S),\delta_1.H=0\}<+\infty.$$
   As a numerical wall of $v$ is given by $\delta_1^\perp$, there are only finitely many of them.
\end{proof}

\subsection{Walls and chambers of a spherical Mukai vector}
From now on we will focus on the case of K3 surfaces and spherical bundles.

 The existence of semistable spherical bundles is due to Kuleshov \cite{Kuleshov:spherical}. 
\begin{Thm}[{\cite[Theorem 10.2.7]{HuybrechtsK3book}}]\label{thm:existsph}
     Let $S$ be a smooth projective K3 surface over $\C$ with an ample line bundle $H$. For any spherical Mukai vector $v$, there exists a $\mu_H$-semistable vector bundle $E$ with $v(E)=v$. If $E$ is $\mu_H$-stable, then it is the unique $\mu_H$-semistable bundle with Mukai vector $v$.\hfill $\Box$
\end{Thm}

\begin{PropDef}\label{propdef:finstabsph}
    Let $S$ be a smooth projective K3 surface with rational polyhedral nef cone and $v$ be a spherical Mukai vector. Then $$\mathsf{H}_S(v):=\#\{E\in\Coh(S)\;|\;v(E)=v,\; E\text{ is }\mu_H\text{-stable for some ample divisor } H.\}$$ is finite.
\end{PropDef}
\begin{proof}
    As $v$ is spherical, it is clear that $(r,\ell)$ is primitive. By Proposition \ref{prop:finwalls}, there are finitely many numerical walls of $v$. As $\Amp(S)$ is convex, it is divided by the numerical walls into finitely many numerical chambers.
    
    For all $H$ and $H'$ in the same numerical chamber, by Lemma \ref{lem:chaningpolarizationsegment} and \ref{lem:ssemiimpliesonwall}, $M^{ss}_H(v)=M^s_H(v)=M^s_{H'}(v)$. By Theorem \ref{thm:existsph}, in each numerical chamber there is a unique slope stable bundle with Mukai vector $v$.

    For each polarization $H$ on the numerical wall of $v$, if $M^s_H(v)=\{E\}$ is non-empty, then by Lemma \ref{lem:chaningpolarizationsegment} and \ref{lem:ssemiimpliesonwall}, the bundle $E$ is slope stable in all neighbor numerical chambers. So slope stable bundles on the numerical wall never contribute to the count. Therefore, $\mathsf H(v)$ is bounded by the number of numerical chambers which is finite.
\end{proof}

\begin{PropDef}\label{propdef:actualwall}
     Let $S$ be a smooth projective K3 surface and $v$ be a spherical Mukai vector. Let $\cW$ be a numerical wall of $v$, then the following are equivalent.
     \begin{enumerate}[(1)]
         \item $\forall\; H\in\cW$, $M^s_H(v)=\emptyset$.
         \item $v=v_1+\dots+v_s$, $s\geq2$, for some spherical Mukai vectors $v_i$ such that $\cW(v,v_i)=\cW$  and $M^s_H(v_i)\neq \emptyset$ for some $H\in\cW$ and all $i$.
         \item $\cW=\cW(v,w)$ for some  spherical Mukai vector $w$ satisfying $\rk(w)<\rk(v)$ and  $\langle w,v\rangle <0$.
         \item[(3)$'$]  $\cW=\cW(v,w)$ for some  spherical Mukai vector $w$ satisfying $\rk(w)\leq\frac{\rk(v)}2$,  $\langle w,v\rangle <0$ and $M^s_H(w)\neq \emptyset$ for some $H\in\cW$.
         \item[(3)$''$] $\cW=\cW(v,w)$ for some  spherical Mukai vector $w$ satisfying $\rk(w)\leq\tfrac{1}{2}\rk(v)$ and  $\langle w,v\rangle <0$.
         
         \item $\forall\; H_\pm\in\Amp(S)$ in two different components of $\Amp(S)\setminus \cW$, there is no spherical bundle $E$ with $v(E)=v$ both $\mu_{H^+}$ and $\mu_{H^-}$-slope stable.
     \end{enumerate}
     We will call such $\cW$ an \emph{actual wall} of $v$ and each connected component of $\Amp(S)\setminus(\cup_{\text{actual walls }}\cW)$ a \emph{chamber} of $v$. We will call a spherical Mukai vector $w$ as that in (3) a \emph{destabilizing factor} of $v$.
\end{PropDef}
\begin{proof}
    (1)$\implies $(2): Let $H$ be an ample divisor on $\cW$ but not on any other numerical walls of $v$. Let $H_+$ be in the neighbor numerical chamber of $H$. By Theorem \ref{thm:existsph}, there exists a $\mu_{H_+}$-stable and $\mu_{H}$-slope semistable spherical  bundle $E$ with Mukai vector $v$. By (1), $E$ is strictly $\mu_H$-semistable. By the same argument as that in \cite[Lemma 6.2]{BM:walls}, the Jordan--H\"older filtration factor $E_i$'s of $E$ are all spherical. Their Mukai vectors give the $v_i$'s satisfying the condition in (2).

    We may write $v=a_1w_1+\dots +a_tw_t$ by combining similar terms of Mukai vectors of Jordan--H\"older factors of $E$. The following lemma will be useful for the rest of the argument. 
    \begin{Lem}\label{lem:410}
        For every $1\leq i\leq t$,
    \begin{align}
        a_i+\langle v,w_i\rangle \geq 0. \label{eq410}
    \end{align}
    \end{Lem}
    \begin{proof}[Proof of the Lemma]
     Denote by $E_i$  the unique object in $M^s_{H_0}(w_i)$. As both $E_i$ and $E$ are $\mu_{H_+}$-slope stable, the spaces $\Hom(E,E_i)$ and $\Hom(E_i,E)$ cannot both be non-zero. If both of them are zero, then $\langle v,w_i\rangle \geq 0$ and \eqref{eq410} holds automatically. 
     
     Without loss of generality, we may assume $\Hom(E_i,E)\neq 0$ and $\Hom(E,E_i)=\Hom(E_i,E[2])=0$. It follows that 
     \begin{align}\label{eq4112}
         \hom(E_i,E)+\langle v_i,v\rangle=\hom(E_i,E)-\chi(E_i,E)=\hom(E_i,E[1])\geq0.
     \end{align}
      We may exhaust the $E_i$ as subfactors in the Jordan--H\"older filtration of $E$. As $\Hom(E_i,E_i[1])=0$, this subfactor is $E_i^{\oplus a}$ for some $a\in\Z_{\geq0}$. It follows that $E$ fits into the short exact sequence  $0\to E_i^{\oplus a}\to E\to E'\to 0$ of bundles for some $\Hom(E_i,E')=0$. In particular, $a_i\geq a=\hom(E_i,E)$. Together with \eqref{eq4112}, the inequality \eqref{eq410} holds.
    \end{proof}

    (1)$\implies$(3)$'$: Note that $\sum_{i=1}^s\langle v_i,v\rangle =\langle v,v\rangle =-2$, there exists $v_i$ such that $\langle v_i,v\rangle<0$. Without the loss of generality, we may assume $i=1$. If $\rk(v_1)\leq \tfrac{1}{2}\rk(v)$, then $v_1$ satisfies all conditions on $w$ in  (3)$'$.
    
    Assume that $\rk(v_1)>\tfrac{1}{2}\rk(v)$, then $v_1\neq v_i$ for any $i\geq 2$.
    
     By Lemma \ref{lem:410},  $\langle v_1,v\rangle=-1$, it then follows that $\sum^s_{i=2}\langle v_i,v\rangle=-1$. There exists $i\geq 2$ such that $\langle v_i,v\rangle<0$. Note that $\rk(v_1)>\tfrac{1}{2}\rk(v)$, so $\rk(v_i)<\tfrac{1}{2} \rk(v)$ and $v_i$ satisfies all conditions on $w$ in  (3)$'$.
    
    (3)$'\implies$(3)$''\implies$(3): Obvious.
    
    (2)$\implies$(3): Note that $\sum_{i=1}^s\langle v_i,v\rangle =\langle v,v\rangle =-2$, there exists $v_i$ such that $\langle v_i,v\rangle<0$.
    
    (3)$\implies$(4): Suppose $E$ with $v(E)=v$ is both $\mu_{H_\pm}$-stable, then by Lemma \ref{lem:chaningpolarizationsegment}, $E$ is $\mu_H$-stable for some $H$ on $\cW$. Let $w$ be a spherical Mukai vector as that in (3). By Theorem \ref{thm:existsph}, there exists a $\mu_H$-semistable spherical bundle $F$ with $v(F)=w$.

    Since $H\in \cW=\cW(v,w)$, the slope $\mu_H(E)=\mu_H(F)$. Since $\chi(F,E)=-\langle w,v\rangle >0$, we have $\Hom(E,F)\neq 0$ or $\Hom(F,E)\neq 0$. As $E$ is a $\mu_H$-stable vector bundle, this can only happen if $E$ is a subbundle or quotient bundle of $F$. However, $\rk(E)=\rk(v)>\rk(w)=\rk(F)$ by assumption, we get the contradiction. So (4) holds.

    (4)$\implies$(1): Suppose $E$ is $\mu_H$-slope stable for some $H\in\cW$, then  by Lemma \ref{lem:chaningpolarizationsegment} and \ref{lem:ssemiimpliesonwall}, the bundle $E$ is slope stable in all neighbor numerical chambers of $H$. This contradicts the assumption of (4). So (1) holds.
    \end{proof}

    \begin{Rem}\label{rem:expressionofsph}
        Note that $v$ might have different decompositions as that in Proposition and Definition \ref{propdef:actualwall}.(2) even with respect to a very general $H\in\cW$. Lemma \ref{lem:410} only holds for the sum of Mukai vectors of Jordan--H\"older factors of a $\mu_{H_+}$-slope stable (or more generally, a $\mu_H$-semistable) $E$. In general, Lemma \ref{lem:410} does not hold for an arbitrary decomposition of $v$ as that in Proposition and Definition \ref{propdef:actualwall}.(2).
    \end{Rem}

\begin{Cor}\label{cor:HvdeterminedbyAmp}
    Let $S$ be a smooth projective K3 surface with a rational polyhedral nef cone and $v=(\rk,\ell,s)$ be a spherical Mukai vector, then $\mathsf H(v)$ is the number of actual chambers.
   
    Let $S'$ be another K3 surface. Assume that there exists an injective group homomorphism $f:\NS(S')\to\NS(S)$ preserving the intersection numbers  and $f(\Nef(S'))\subseteq \Nef(S)$. Then $\mathsf H_{S'}((\rk,\ell,s))\leq \mathsf H_S((\rk,f(\ell),s))$. In particular,  $\mathsf H_S(v)$ only depends on $(\Nef(S),(-,-)_{\text{intersection}})$.
\end{Cor}
\begin{proof}
    The first statement follows by  Proposition and Definition \ref{propdef:actualwall}.(4).   For the second statement, by  Proposition and Definition \ref{propdef:actualwall}.(3), all actual walls of $v$ in $\NS(S)$ are mapped into actual walls of $f(v):=(\rk,f(\ell),s)$ in $\NS(S')$. By Proposition and  Definition \ref{propdef:actualwall}.(4), two different actual chambers of $v$ in $\NS(S)$ are still separated by some actual wall in $\NS(S')$. The conclusion follows. 
\end{proof}

We finish the section with a remark on the difference between Gieseker and slope stabilities.

\begin{Rem}\label{rem:slopevspoly}
    Let $S$ be a smooth projective surface and  $v=(\rk,\ell,s)\in\tilde{\mathrm{H}}(S,\Z)$ with $\rk>0$ and $(\rk,\ell)$ being primitive. For $H\in\Amp_\Q(S)$, we denote by $M^s_{H,\text{slope}}(v)$ and $M^s_{H,\text{Gieseker}}(v)$ the space of all $\mu_H$-slope stable and $p_H$-stable coherent sheaves respectively.
    \begin{enumerate}
        \item  If $H$ is in an open numerical chamber, then we always have $M^s_{H,\text{slope}}(v)=M^s_{H,\text{Gieseker}}(v)$.\\
         If $F$ is $\mu_H$-slope stable, then $F$ is $\mu_{H'}$-slope stable for $H'$ in an open neighbourhood of $H$. This is not the case for $p_H$-stability. See Example \ref{eg:lonelystable} below for a spherical bundle that is $p_H$-stable only for $mH$ but not with respect to any other polarization.
        \item If $H$ and $H'$ are both in $\cap_{i=1}^s\cW_i$ for some numerical chambers $\cW_i$ of $v$ and the line segment $\{(1-t)H+tH'\;|\;t\in[0,1]\}$ does not intersect any other numerical chamber, then $M^s_{H,\text{Gieseker}}(v)=M^s_{H',\text{Gieseker}}(v)$. 
        In particular, if $S$ is a K3 surface with rational polyhedral nef cone and $v$ is a spherical Mukai vector, then $\mathsf H_{\text{Gieseker}}(v)$ is also finite. It is clear that $\mathsf H_{\text{Gieseker}}(v)\geq \mathsf H(v)$. The value $\mathsf H_{\text{Gieseker}}(v)$ will be strictly greater than $\mathsf H(v)$ for some $v$ if isolated $p_H$-stable sheaves as that in Example \ref{eg:lonelystable} appear on the numerical wall. 
    \end{enumerate}
\end{Rem}

\begin{Ex}\label{eg:lonelystable}
    Let $S$ be a generic elliptic K3 surface with a section and adopt the notation as that in Section \ref{sec5}. We may consider the spherical bundle $B:=\mathsf T_{\cO_S(E-F)}(\mathsf T_{\cO_S(F-E)}(\cO_S(2E-2F)))$, which is with Mukai vector $v(B)=(113,82e-82\sigma,-119)$. Then $B$ is  $p_{3e+\sigma}$-stable, but $B$ is not slope stable with respect to any polarization. The spherical bundle $B$ is not Gieseker stable with respect to any polarization other than $3e+\sigma$ up to a scalar.
\end{Ex}
\section{Case study: counting stable spherical bundles on an elliptic K3 surface}\label{sec5}

Let $S$ be a generic elliptic K3 surface  $\pi:S\to \P^1$ with a section $C\subset S$. Denote by  $E=\pi^{-1}(p)$ a generic fiber, $\sigma:=[C]$, and $e:=[E]$. Assume that $\NS(S)=\Z e\oplus \Z\sigma$. The intersection numbers are given as
\begin{align}
    e^2=0 \;\;\; e.\sigma=1 \;\;\; \sigma^2=-2.
\end{align}
The cones are $\Nef(S)=\R_{\geq 0}.e+\R_{\geq 0}.(2e+\sigma)$; $\Peff(S)=\R_{\geq 0}.e+\R_{\geq 0}.\sigma$.

\subsection{Accurate counting}

In this section, we compute the values of $\mathsf H_S(v)$ for such an elliptic K3 surface $S$. We first give a simple parametrization for all spherical Mukai vectors of $S$.
\begin{Lem}\label{lem:vforsphericalvector}
    Let $S$ be an elliptic K3 surface as above, then there is a one-to-one correspondence:
    \begin{align*}
        v:\Q\times \Z & \to \{\text{spherical Mukai vectors in } \Z\oplus \NS(S)\oplus \Z \} \\
        (\tfrac{n}m,k) &\mapsto \left(m,n\sigma + (km+n-n^{\varphi(m)-1}) e,\frac{-n^2+n(km+n-n^{\varphi(m)-1})+1}m\right).
    \end{align*}
\end{Lem}

Here $\tfrac{n}m$ is in the unique irreducible form in the sense that $m\geq 1$ and $\gcd(n,m)=1$. We write $\varphi(-)$ for the Euler's totient function. In particular,  $n^{\varphi(m)}\equiv 1(\mod m)$.  When $m=1$ and $n=0$, we set the term $n^{\varphi(m)-1}$ to be $0$.
\begin{proof}
   By definition, the rank $m>0$. By a direct computation, $\langle v(\frac{n}m,k),v(\frac{n}m,k)\rangle =-2$. So $v$ is well-defined. 

   As every rational number has a unique irreducible form, the map $v$ is clearly injective.

   For every spherical Mukai vector $v=(m,n\sigma +be,s)$,  by definition the rank $m>0$. The self-intersection is  \begin{align}
       -2=\langle v,v\rangle = -2n^2+2bn-2ms.\label{eq5.2}
   \end{align}
   It follows that $m|-n^2+bn+1$. So $\gcd(m,n)=1$ and $$bn\equiv n^2-1(\mod m)\implies b\equiv n-n^{\varphi(m)-1}(\mod m).$$
    So there exists $k\in \Z$ such that $b=km+n-n^{\varphi(m)-1}$. By \eqref{eq5.2}, the last factor $s=\tfrac{1}{m}(-n^2+bn+1)$. Therefore, the map $v$ is surjective.   
\end{proof}

Note that if a spherical bundle $E$ is $\mu_H$-stable then the dual bundle $E^*$ and $E\otimes\cL$ are $\mu_H$-stable for every line bundle $\cL$.  It follows that $\mathsf H_S(v(a+n,m))=\mathsf H_S(v(a,0))=\mathsf H_S(v(-a,0))$ for all $a\in \Q$, and $n,m\in\Z$. 
\begin{Not}\label{not:H}
    By abuse of notions, we will write the counting function $\mathsf H(a):=\mathsf H_S(v(a,0))$ for every $a\in \Q$. It is clear that $\mathsf H(a+n)=\mathsf H(a)=\mathsf H(-a)$ for every $n\in \Z$. 
\end{Not}

We can describe the destabilizing factor as follows.

\begin{Lem}\label{lem:desfactor}
    A spherical Mukai vector $v(\tfrac{s}{r},k)$ is a destabilizing factor of $v(\tfrac{n}{m},0)$ if and only if $r<m$ and 
    \begin{align}\label{eq513}
        (\frac{n}m-\frac{s}r)^2-\frac{1}{r^2}< (\frac{n}m-\frac{s}r)(\frac{n-n^{\varphi(m)-1}}m-\frac{s-s^{\varphi(r)-1}}r-k)<0
    \end{align}
\end{Lem}
\begin{proof}
    By Proposition and Definition \ref{propdef:actualwall}.(3), $v(\tfrac{s}{r},k)$ is a destabilizing factor of $v(\tfrac{n}{m},0)$ if and only if $r<m$ and 
    \begin{enumerate}
        \item  $\cW(v(s/r,k),v(n/m,0))\cap \Amp(S)\neq \emptyset$;
        \item $\langle v(s/r,k),v(n/m,0)\rangle\leq -1$ or equivalently $<0$.
    \end{enumerate}
    Denote by divisor $D_1=n\sigma +(n-n^{\varphi(m)-1})e$ and $D_2=s\sigma+(kr+s-s^{\varphi(r)-1})e$. Then condition (a) $\iff$ the divisor of  $rD_1-mD_2$    is not in $\Peff(S)$ or $-\Peff(S)$. 
    \begin{align*}
        \iff (rn-ms)(r(n-n^{\varphi(m)-1})-m(kr+s-s^{\varphi(r)-1}))<0.
    \end{align*}
    Condition (b) 
    \begin{align*}
        \iff & D_1.D_2-r\frac{D_1^2+2}{2m}-m\frac{D_2^2+2}{2r}<-\frac{r}m \\
        \iff & 2rmD_1.D_2-r^2D^2_1-2r^2-m^2D_2^2-2m^2< -2r^2\\
        \iff & -(rD_1-mD_2)^2<2m^2.\\
        \iff & 2(rn-ms)^2 - 2(rn-ms)\left(r(n-n^{\varphi(m)-1})-m(kr+s-s^{\varphi(r)-1})\right)<2m^2.
   \end{align*}
   Combine this together with condition (a), dividing all terms by $r^2m^2$, this is equivalent to the inequality \eqref{eq513}.
\end{proof}

\begin{Ex}\label{eg:computeH}
Note that \eqref{eq513} in particular implies $|\frac{n}m-\frac{s}r|<\frac{1}r$ and $$\left|\frac{n-n^{\varphi(m)-1}}m-\frac{s-s^{\varphi(r)-1}}r-k\right|<\frac{m}r.$$
For each given $\frac{n}m$, these two conditions reduce the number of possible destabilizing factors to around $2m\ln{m}$. With the help of the computer, we may list all destabilizing factors $v(\frac{s}r,k)$ for $v(\frac{n}m,0)$ when $m$ is small. Note that the wall $\cW(v(\frac{s}r,k),v(\frac{n}m,0))$ is given by
\begin{align}\label{eq:wallelliptic}
    ((rn-ms).\sigma+(r(n-n^{\varphi(m)-1})-m(kr+s-s^{\varphi(r)-1})).e)^\perp\cap \Amp(S)
\end{align}
which is determined by the ratio of two coefficients. We may then compute the number of all actual walls. 
We list a few explicit computations here.\\
 
$\mathsf H(1)=1$; $\mathsf H(\frac{1}{2})=2$;  $\mathsf H(\frac{1}{3})=3$;  $\mathsf H(\frac{1}{4})=4$;  $\mathsf H(\frac{1}{5})=5$;  $\mathsf H(\frac{2}{5})=6$;  $\mathsf H(\frac{1}{6})=6$;  $\mathsf H(\frac{1}{7})=7$;  $\mathsf H(\frac{2}{7})=7$;  $\mathsf H(\frac{3}{7})=7$;  $\mathsf H(\frac{1}{8})=8$;  $\mathsf H(\frac{3}{8})=6$;  $\mathsf H(\frac{1}{9})=9$;  $\mathsf H(\frac{2}{9})=7$;  $\mathsf H(\frac{4}{9})=8$;  $\mathsf H(\frac{1}{10})=10$;  $\mathsf H(\frac{3}{10})=9$; 

$\dots\dots$

$\mathsf H(\frac{1}{21})=21$;  $\mathsf H(\frac{2}{21})=13$;  $\mathsf H(\frac{4}{21})=12$;  $\mathsf H(\frac{5}{21})=13$;  $\mathsf H(\frac{8}{21})=11$;  $\mathsf H(\frac{10}{21})=14$;  

$\dots\dots$

$\mathsf H(\frac{1}{32})=32$;  $\mathsf H(\frac{3}{32})=16$;  $\mathsf H(\frac{5}{32})=16$;  $\mathsf H(\frac{7}{32})=14$;  $\mathsf H(\frac{9}{32})=15$;  $\mathsf H(\frac{11}{32})=16$;  $\mathsf H(\frac{13}{32})=14$;  $\mathsf H(\frac{15}{32})=13$;  

$\dots\dots$

$\mathsf H(\frac{1}{93})=93$;  $\mathsf H(\frac{2}{93})=49$;  $\mathsf H(\frac{4}{93})=30$;  $\mathsf H(\frac{5}{93})=28$;  $\mathsf H(\frac{7}{93})=25$;  $\mathsf H(\frac{8}{93})=25$;  $\mathsf H(\frac{10}{93})=26$;  $\mathsf H(\frac{11}{93})=24$;  $\mathsf H(\frac{13}{93})=23$;  $\mathsf H(\frac{14}{93})=24$;  $\mathsf H(\frac{16}{93})=21$;  $\mathsf H(\frac{17}{93})=23$;  $\mathsf H(\frac{19}{93})=22$;  $\mathsf H(\frac{20}{93})=22$;  $\mathsf H(\frac{22}{93})=22$;  $\mathsf H(\frac{23}{93})=31$;  $\mathsf H(\frac{25}{93})=23$;  $\mathsf H(\frac{26}{93})=24$;  $\mathsf H(\frac{28}{93})=20$;  $\mathsf H(\frac{29}{93})=25$;  $\mathsf H(\frac{32}{93})=20$;  $\mathsf H(\frac{34}{93})=24$;  $\mathsf H(\frac{35}{93})=21$;  $\mathsf H(\frac{37}{93})=26$;  $\mathsf H(\frac{38}{93})=22$;  $\mathsf H(\frac{40}{93})=23$;  $\mathsf H(\frac{41}{93})=23$;  $\mathsf H(\frac{43}{93})=26$;  $\mathsf H(\frac{44}{93})=25$;  $\mathsf H(\frac{46}{93})=50$;  

$\dots\dots$

$\mathsf H(\frac{92}{811})=48$;  $\mathsf H(\frac{93}{811})=47$;  $\mathsf H(\frac{94}{811})=46$;  $\mathsf H(\frac{95}{811})=43$;  $\mathsf H(\frac{96}{811})=45$;  $\mathsf H(\frac{97}{811})=50$;  $\mathsf H(\frac{98}{811})=48$;  $\mathsf H(\frac{99}{811})=46$;  $\mathsf H(\frac{100}{811})=53$; $\dots$; $\mathsf H(\frac{266}{811})=50$;  $\mathsf H(\frac{267}{811})=51$;  $\mathsf H(\frac{268}{811})=58$;  $\mathsf H(\frac{269}{811})=83$;  $\mathsf H(\frac{270}{811})=276$;  $\mathsf H(\frac{271}{811})=146$;  $\mathsf H(\frac{272}{811})=72$;  $\mathsf H(\frac{273}{811})=60$;  $\mathsf H(\frac{274}{811})=60$;  $\mathsf H(\frac{275}{811})=47$;  $\mathsf H(\frac{276}{811})=46$;  $\mathsf H(\frac{277}{811})=45$. 
   
\end{Ex}

We may prove the computation of $\mathsf H(\tfrac{n}{m})$ for some explicit value of $\frac{n}{m}$ as follows.
\begin{Ex}\label{eg:basiccomputation}
Let $S$ be a generic elliptic K3 surface admitting a section and adopt Notation \ref{not:H}, then
    \begin{enumerate}
        \item $\mathsf H(\frac{1}m)=m$.
        \item $\mathsf H(\frac{n}{2n+1})=n+4$ when $n\geq 2$.
    \end{enumerate}
\end{Ex}
\begin{proof}
    (a) By Proposition and Definition \ref{propdef:actualwall}.(3)$''$, we only need to consider all $v(\frac{s}r,k)$ with $r\leq \frac{m}2$ satisfying \eqref{eq513}. When $r\geq 2$, as $|\frac{1}m-\frac{1}r|<\frac{1}{r^2}$ and $\gcd(s,r)=1$, the destabilizing factor is in the form of $v(\tfrac{1}{r},k)$.

    Note that both $n=s=1$ in \eqref{eq513}, so by the second inequality, $k\leq -1$. By the first inequality, $(\tfrac{1}m-\tfrac{1}r)^2-\tfrac{1}{r^2}<(\tfrac{1}m-\tfrac{1}r)(-k)\leq -\tfrac{1}{2r}\leq -\tfrac{1}{r^2}$. So there is no solution to \eqref{eq513} when $r\geq 2$.\\

    We only need to consider all destabilizing factors in the form of $v(s,k)$. It is clear that $s=0$ or $1$. Inequalities in \eqref{eq513} become $(\tfrac{1}{m}-s)^2-1< (\tfrac{1}m-s)(-k)<0$,
    
    When $s=1$, there is no solution of $k$ to \eqref{eq513}.

    When $s=0$, solutions to \eqref{eq513} are $k=1,\dots, m-1$.

    Note that the slopes of the lines through $(\frac{1}m,0)$ and each $(0,k)$ are all different. So the $m-1$ destabilizing factors of $v(\frac{1}m,0)$ correspond to $m-1$ different actual wall. So $\mathsf H(\frac{1}m)=m$.

    (b) Note that $n^{\varphi(2n+1)-1}\equiv -2(\mod 2n+1)$. Replacing $v(\tfrac{n}{2n+1},0)$ by $v(\frac{n}{2n+1},k_n')$, we may assume $v(\tfrac{n}{2n+1},k_n')=(2n+1,n.\delta+(n+2).e,*)$. In \eqref{eq513}, the term $\frac{n-n^{\varphi(2n+1)-1}}{2n+1}$ is accordingly replaced as $\frac{n+2}{2n+1}$.
    
    When $r=1$, then $s=0$ or $1$. Solutions  to \eqref{eq513} are $(s,k)=$ $(0,1)$, $(0,2)$, or $(1,0)$.

    When $r=2$, then $s=1$. Solutions to \eqref{eq513} are $k=0,-1,\dots,-n+1$.

    When $r=2d$ with $2\leq d\leq n/2$, then as $|\frac{n}{2n+1}-\frac{s}{2d}|<\frac{1}{2d}$, the numerator $s$ can only be $d-1$. It follows that $2|d$, so $(d-1)^{\varphi(2d)-1}\equiv -d-1 (\mod 2d)$. It is clear that \eqref{eq513} has no solution for $k\in\Z$.

    When $r=2d+1$ with $1\leq d\leq n/2$, then $s$ can only be $d$. There is a unique $k_d\in\Z$ satisfying $\eqref{eq513}$. More precisely,  $v(\tfrac{d}{2d+1},k_d)=(2d+1,d.\delta+(d+2).e,*)$. Note that the wall $\cW(v(\tfrac{n}{2n+1},k_n'),v(\tfrac{d}{2d+1},k_d))$ is always spanned by $\delta+5e$ which is the same as that given by $\cW(v(\tfrac{n}{2n+1},k_n'),v(0,2))$. 

    To sum up,  actual walls correspond to the destabilizing factors $v(0,1)$, $v(0,2)$, $v(1,0)$, $v(\tfrac{1}{2},k)$ where $k=0,-1,\dots,-n+1$. It is clear that these walls are all different from each other when $n\geq 2$. So there are $n+4$ chambers in total for $v(\tfrac{n}{2n+1},k_n')$.
\end{proof}
\begin{Rem}
    There are some other patterns of $\mathsf H(\tfrac{n}m)$. We write down some typical ones in the remark. Here we adopt the notion of continuous fraction  $$[a_1,a_2,\dots,a_s]:=\tfrac{1}{a_1+\frac{1}{a_2+\frac{1}{\dots+\frac{1}{a_s}}}}.$$
    \begin{enumerate}
    \item $\mathsf H(\tfrac{n}m)\leq m$ when $m\geq 6$. 
    \item $\mathsf H(\tfrac{2}{2n+1})=\begin{cases}
        n+4 & \text{ when $n\equiv 0,1 (\mod 3)$;}\\
        n+3 & \text{ when $n\equiv 2 (\mod 3)$.}
    \end{cases}$  
    \item $\mathsf H([a,b])=\begin{cases}
        b+2a & \text{when $b\geq a \geq 2$;}\\
        a+2b \text{ or } a+2b-1& \text{when $a\geq b\geq 2$.}
    \end{cases}$
    \item$\mathsf H([a_1,\dots,a_s,N])=N+2\sum a_i$ when $a_1\geq 2$ and $N$ is sufficiently large with respect to all $a_i$'s.
    \item Let $b_1=b_2=1$ and $b_{n+1}:=b_n+b_{n-1}$, then $\mathsf H(\frac{b_n}{b_{n+1}})=\mathsf H([\underbrace{1,1,\dots,1}_{n^{\text{th}}\text{ 1 in total}}])=\lfloor\frac{n}2\rfloor^2+2$.
    \end{enumerate}

\subsection{Asymptotic behaviors}
    We also notice the following general behavior of $\mathsf H(\frac{n}m)$ when $m$ tends to infinity:
    \begin{align}\label{eq511}
       \mathsf H_{\min}(m):= \min_{1\leq n\leq m,\;\gcd(n,m)=1}\{\mathsf H(\tfrac{n}m)\}= \Theta((\ln m)^2).\\
       \mathsf H_{\mathrm{ave}}(m):= \mathrm{Average}_{1\leq n\leq m,\;\gcd(n,m)=1}\{\mathsf H(\tfrac{n}m)\}= \Theta((\ln m)^2).\label{eq5116}
    \end{align}
\end{Rem}
 In the rest of the note, we prove the following weaker version of the average estimation \eqref{eq5116} on the counting $\mathsf H(v)$.
\begin{Prop}\label{prop:estimationofH}
    Let $S$ be a generic elliptic K3 surface admitting a section.  Adopt Notation \ref{not:H}, then for any $\alpha>0$, there exists constant $C_1$ and $C_2$ such that for $m\gg0$,
    \begin{align}\label{eq512}
      \tfrac{C_1}m(\varphi(m)\ln m)^2\leq  \sum_{\substack{1\leq n< m,\; \gcd(n,m)=1}}\mathsf H\left(\tfrac{n}m\right)\leq C_2 m^{1+\alpha}.
    \end{align}
\end{Prop}

We make some notions and lemmas on elementary number theory before the proof. For a pair of rational numbers $\frac{n}m$ and $\frac{s}r$ with  $1\leq r<m$, we will write:
\begin{align}
    \mathsf F\left(\tfrac{n}{m},\tfrac{s}{r}\right):=\#\{k\in \Z\;|\; k \text{ satisfies inequalities in \eqref{eq513}}\}.
\end{align}

\begin{Lem}\label{lem:boundonF}
The function $\mathsf F$ and $\mathsf H$ satisfy the following properties.

\begin{enumerate}[(1)]
    \item $\mathsf F\left(\tfrac{n}{m},\frac{s}{r}\right)=\mathsf F\left(1-\frac{n}{m},1-\frac{s}{r}\right)$.\label{lem5q1}
    \item When $|rn-ms|\geq m$, $\mathsf F\left(\frac{n}{m},\frac{s}{r}\right)=0$.\\
When $|rn-ms|< m$, $\mathsf F\left(\frac{n}{m},\frac{s}r\right)\geq \left\lfloor\frac{m^2-(rn-ms)^2}{mr|rn-ms|}\right\rfloor$.
    \label{lem5q2}
    \item For every $r<m$ and $\frac{s}r\in \Q$, $\mathsf H\left(\frac{n}{m}\right)\geq \mathsf F\left(\frac{n}{m},\frac{s}{r}\right)+1$.\label{lem5q3}
    \item \label{lem5q4}Assume that $0<n<m$, then 
    \begin{align}\label{eq56}
        \mathsf H\left(\tfrac{n}{m}\right)\leq 1+ \sum_{1\leq r\leq \left\lfloor \tfrac{m}{2}\right\rfloor}\sum_{\substack{0\leq s\leq r \\ \gcd(s,r)=1}}\mathsf F\left(\tfrac{n}{m},\tfrac{s}{r}\right).
    \end{align}
    \item \label{lem5q5} Let $a=\gcd(m,r)$, then $\mathsf F\left(\frac{n}{m},\frac{s}{r}\right)\leq \mathsf F\left(\frac{an}{m},\frac{as}{r}\right)$.
    \item \label{lem5q6}When $\gcd(m,r)=1$, 
    \begin{align}\label{eg:F}
        \mathsf F\left(\tfrac{n}{m},\tfrac{s}{r}\right)=\#\left\{t\in\Z_{\geq 1}\mid (rn-ms)\vert m^2+r^2-tmr >(rn-ms)^2\right\}.
    \end{align}
\end{enumerate}
\end{Lem}
\begin{proof}
    Statement \eqref{lem5q1} follows by \eqref{eq513} directly.
    
    \eqref{lem5q2} When $|rn-ms|\geq m$, the first formula $\left(\frac{n}{m}-\frac{s}{r}\right)^2-\frac{1}{r^2}\geq 0$. So there is no $k$ satisfying \eqref{eq513}.

    When $|rn-ms|< m$, the inequality follows by dividing $\left|\frac{n}m-\frac{s}r\right|$ on \eqref{eq513}.
    
    \eqref{lem5q3} By Lemma \ref{lem:desfactor}, there are $\mathsf F\left(\frac{n}{m},\frac{s}{r}\right)$ different $k_i$'s such that $v(\frac{s}r,k_i)$ destabilizes $v(\frac{n}m,0)$. By \eqref{eq:wallelliptic}, the walls of them are different from each other. By Proposition and Definition \ref{propdef:actualwall}, the inequality holds.

    \eqref{lem5q4} By Proposition and Definition \ref{propdef:actualwall}.(3)$''$,  each actual wall of $v(\frac{n}m,0)$ corresponds to at least one destabilizing factor $v(\frac{s}r,k)$ with $r\leq \frac{m}2$. When $\frac{s}r\not \in[0,1]$, $|\frac{s}{r}-\frac{n}{m}|\geq \frac{1}r$. By \eqref{lem5q1}, $\mathsf F\left(\frac{n}{m},\frac{s}{r}\right)=0$. The inequality \eqref{eq56} follows.

    \eqref{lem5q5} Assume $m=am'$ and $r=ar'$, then $m'|n^{\varphi(m)-1}-n^{\varphi(m')-1}$ and $r'|s^{\varphi(r)-1}-s^{\varphi(r')-1}$. Denote by $c=\frac{n^{\varphi(m)-1}-n^{\varphi(m')-1}}{m'}-\frac{s^{\varphi(r)-1}-s^{\varphi(r')-1}}{r'}\in\Z$. 

    For every $k\in\Z$ satisfying \eqref{eq513} with respect to $(\frac{n}{m},\frac{s}{r})$, the integer $c+ka$ satisfies \eqref{eq513} with respect to $(\frac{n}{m'},\frac{s}{r'})$.

    \eqref{lem5q6} An integer $k$ satisfies \eqref{eq513}
    \begin{align*}
        \iff & 
            (rn-ms)^2-m^2<(rn-ms)(r(n-n^{\varphi(m)-1})-m(s-s^{\varphi(r)-1})-kmr)< 0 \\
        \iff & m^2> (rn-ms)(rn^{\varphi(m)-1}-ms^{\varphi(r)-1}+kmr)>(rn-ms)^2.
    \end{align*}
    When $\gcd(m,r)=1=\gcd(n,m)=\gcd(s,r)$, it is clear that $\gcd(rn-ms,mr)=1$. 

    Note that $$(rn-ms)(rn^{\varphi(m)-1}-ms^{\varphi(r)-1})\equiv r^2n^{\varphi(m)}+m^2s^{\varphi(r)}\equiv r^2+m^2(\mod mr).$$
    By SunZi Theorem, 
    \begin{align*}
        & \{a\in\Z\mid a\equiv m^2+r^2(\mod mr), (rn-ms)\vert a\}\\
        =& \{(rn-ms)(rn^{\varphi(m)-1}-ms^{\varphi(r)-1})+k(rn-ms)mr\mid k\in \Z\}.
    \end{align*}
    So \begin{align*}
       & \mathsf F\left(\tfrac{n}{m},\tfrac{s}{r}\right) \\
       =&\#\left\{k\in\Z \;\middle|\; (rn-ms)(rn^{\varphi(m)-1}-ms^{\varphi(r)-1})+k(rn-ms)mr\in \left((rn-ms)^2,m^2\right)\right\}\\
       =&\#\left\{a\in\Z\;\middle|\; a\equiv m^2+r^2(\mod mr), (rn-ms)\vert a, a\in \left((rn-ms)^2,m^2\right) \right\}\\
       =& \#\left\{t\in\Z\;\middle|\; (rn-ms)\vert a=m^2+r^2-tmr \in \left((rn-ms)^2,m^2\right)\right\}\\
       =& \#\left\{t\in\Z_{\geq 1}\;\middle|\; (rn-ms)\vert m^2+r^2-tmr >(rn-ms)^2\right\}.
    \end{align*}
    The claim follows.
\end{proof}

For a pair of positive integers $(m,r)$ satisfying $m>r$, we denote by 

\begin{align*}S_{m,r}:=&\left\{(n,s)\in\Z^2\;\middle|\; 1\leq n\leq m-1,0\leq s\leq r,\gcd(n,m)=\gcd(s,r)=1\right\}.\\
    A_{m,r}:=&\{m^2+r^2-trm\in \Z\cap [1,m^2]\mid t\in \Z\}.\\
    \mathsf G(m,r):=&2\sum_{a\in A_{m,r}} \left\lfloor\tfrac{\tau(a)}2\right\rfloor \text{ and } \mathsf G(m):=\sum_{\substack{1\leq r<m,\;\gcd(r,m)=1}}\mathsf G(m,r),
\end{align*}
where $\tau(n)=\sum_{d\vert n} 1$ is the divisor function.

\begin{Lem}\label{lem:boundonHandF}
    Let $m\geq 2$ be a positive integer. 
    \begin{enumerate}[(1)]
        \item For all positive integers $r,d\in [1,m-1]$, we have
        \begin{align}\label{eq534}
            &\#\left\{(n,s)\in S_{m,r} \mid   rn-ms=d  \right\} \\
             & \begin{cases}
              =1 & \text{if }\gcd(m,r)=1 \text{ and }\gcd(mr,d)=1;  \\
              \leq \varphi(m)/\varphi(\tfrac{m}{a}) & \text{if }\gcd(m,r)=a,\; a\vert d, \text{ and }\gcd(\tfrac{mr}{a^2},\tfrac{d}{a})=1; \\ 
              =0 & \text{otherwise.}
            \end{cases}\notag
        \end{align}
    \item For every $r<m$ with $\gcd(m,r)=1$, 
    \begin{align}\label{eq5132}
       \sum_{(n,s)\in S_{m,r}}\mathsf F\left(\tfrac{n}{m},\tfrac{s}{r}\right)=  \mathsf G(m,r).
    \end{align}
    \item  \begin{align}\label{eqestofHtotal}
        \sum_{\substack{1\leq n\leq m-1\\\gcd(n,m)=1}}\mathsf H\left(\tfrac{n}{m}\right)\leq \varphi(m)+ \sum_{d>1, d\vert m} \frac{\varphi(m)}{\varphi(d)}\mathsf G\left(d\right).
    \end{align}
    \end{enumerate}
\end{Lem}
\begin{proof}
 (1) Assume that $\gcd(m,r)=a$, as  $\gcd(n,m)=\gcd(s,r)=1$, it is  clear that $a|rn-ms$ and $\gcd(\tfrac{mr}{a^2},\tfrac{rn-ms}a)=1$. So the `otherwise' case in the equation holds.

 When $r=1$ and $(n,s)\in S_{m,r}$, it is clear that $n-ms=d$ has a unique solution $(n,s)=(d,0)$. We may assume $r\geq 2$ for the rest of the argument.

Assume that $\gcd(m,r)=1$. By SunZi Theorem, for every $d$ satisfying $\gcd(mr,d)=1$, there is a unique pair $(n,s)\in S_{m,r}$ such that $rn-ms\equiv d (\mod mr)$.  Note that $rn-ms\in [r+m-rm,rm-r]$, if  $rn-ms\equiv d(\mod mr)$ for some $d\in[1,m-1]$, then $rn-ms= d$. The first equality in the equation holds.

Now assume that $\gcd(m,r)=a>1$ and let $m=am'$ and $r=ar'$. For a given number $d=ad'\in[a,m-a]$,  let $(n_0,s_0)$ be the smallest solution to $rn-ms=ad'$ for all $(n,s)\in S_{m,r}$. Then other solutions to $rn-ms=ad'$ are all in the form of $(n_0+tr',s_0+tm')$, where $t=0,1,\dots,a-1$. Note that there are exactly $\tfrac{\varphi(am')}{\varphi(m')}$ many $t$ satisfying $\gcd(s_0+tm',a)=1$, so there are at most $\tfrac{\varphi(am')}{\varphi(m')}$ many pairs of $(n_0+tr',s_0+tm')$ in $S_{m,r}$. The second inequality \eqref{eq534} holds. \\

(2) By Lemma \ref{lem:boundonF}.\eqref{lem5q1}\eqref{lem5q2}, the left-hand side of \eqref{eq5132} equals to 
$$2\cdot\sum_{\substack{(n,s)\in S_{m,r}\; 1\leq rn-ms \leq m-1}}\mathsf F\left(\tfrac{n}{m},\tfrac{s}{r}\right).$$
By (1) and Lemma \ref{lem:boundonF}.\eqref{lem5q6}, this is equal to 
$$2\cdot\sum_{1\leq d \leq m-1,\;\gcd(d,mr)=1}\#\left\{k\in\Z_{\geq 1}\mid d\vert m^2+r^2-kmr >d^2\right\}.$$
As $\gcd(m,r)=1$, it is clear that $\gcd(mr,m^2+r^2-kmr)=1$. In particular, any divisor of $m^2+r^2-kmr$ is coprime to $mr$.

As $m>r$ and $k\geq 1$, we have $m^2+r^2-kmr<m^2$. The equation holds.\\

(3) By Lemma \ref{lem:boundonF}.\eqref{lem5q4}, the left-hand side of \eqref{eqestofHtotal} is less than or equal to 
\begin{align}
 \notag  &  \varphi(m) +\sum_{1\leq r\leq \left\lfloor \tfrac{m}{2}\right\rfloor}\sum_{(n,s)\in S_{m,r}}\mathsf F\left(\tfrac{n}{m},\tfrac{s}{r}\right)\\
   \leq & \varphi(m) +\sum_{1\leq a<m,\;a\vert m}\sum_{\substack{1\leq r<m,\\\gcd(r,m)=a}}\sum_{(n,s)\in S_{m,r}}\mathsf F\left(\tfrac{n}{m},\tfrac{s}{r}\right) \label{eq5113}
\end{align}
Assume $\gcd(m,r)=a$, let $m=am'$ and $r=ar'$. Then by (1), (2), and Lemma \ref{lem:boundonF}.\eqref{lem5q5}, the sum 
\begin{align*}
    \sum_{(n,s)\in S_{m,r}}\mathsf F\left(\tfrac{n}{m},\tfrac{s}{r}\right)& \leq \frac{2\varphi(m)}{\varphi(m')} \sum_{\substack{(n',s')\in S_{m',r'}\\1\leq r'n'-m's' \leq m'-1}}\mathsf F\left(\tfrac{n'}{m'},\tfrac{s'}{r'}\right)=\frac{\varphi(m)}{\varphi(m')} \mathsf G(m',r').
\end{align*}
Substitute this back to \eqref{eq5113}, we get
\begin{align*}
    \varphi(m) +\sum_{1\leq a<m,\;a\vert m}\sum_{\substack{1\leq r<m,\\\gcd(r,m)=a}}\frac{\varphi(m)}{\varphi(m')} \mathsf G(m',r')=\varphi(m) +\sum_{1\leq a<m,\;a\vert m}\frac{\varphi(m)}{\varphi(m')} \mathsf G(m').
\end{align*}
The inequality follows.
\end{proof}

\begin{proof}[Proof of Proposition \ref{prop:estimationofH}]
    We first prove the upper bound. Note that for any $\epsilon>0$, $\tau(n)=o(n^\epsilon)$, it follows that $\mathsf G(m)=m\log m \cdot o(m^\epsilon)=o(m^{1+2\epsilon})$. When $d\vert m$, we have $\frac{\varphi(m)}{\varphi(d)}\leq \frac{m}d$.   The upper bound follows from \eqref{eqestofHtotal}.\\

    We then prove the lower bound. Denote by 
    $$B_m:=\bigcup_{\substack{1\leq r <m^{\frac{1}{4}},\; 0\leq s\leq r-1\\ \gcd(s,r)=1}}\left(\tfrac{s}r,\;\tfrac{s}r+\tfrac{1}{\sqrt{m}}\right).$$
    
    Note that when $r<m^{\frac{1}{4}}$, the intervals $(\frac{s}r,\;\frac{s}r+\frac{1}{\sqrt m})$ do not intersect with each other. So for every $\frac{n}m\in B_m$, there is a unique $\frac{s}r$, denote as $q(\frac{n}m)$, associated with $\frac{n}m$ such that $\tfrac{n}{m}\in\left(\tfrac{s}r,\;\tfrac{s}r+\tfrac{1}{\sqrt{m}}\right)$.  
    
     {\allowdisplaybreaks
     \begin{align*}
     & \sum_{\substack{1\leq n\leq m-1\\\gcd(n,m)=1}}\mathsf H\left(\tfrac{n}{m}\right)\geq  \sum_{\frac{n}m\in B_m}\mathsf H\left(\tfrac{n}{m}\right) 
     \geq \sum_{\frac{n}m\in B_m} 1+\mathsf F\left(\tfrac{n}{m},q\left(\tfrac{n}m\right)\right) & \text{by Lemma \ref{lem:boundonF}.\eqref{lem5q3}}  \\ 
     \geq & \sum_{\frac{n}m\in B_m,\; q(\frac{n}m)=\frac{s}r} \frac{m^2-(rn-ms)^2}{mr|rn-ms|} & \text{by Lemma \ref{lem:boundonF}.\eqref{lem5q2}}\\
     = & \sum_{1\leq r< m^{1/4}}\sum_{\substack{0\leq s\leq r-1\\\gcd(s,r)=1}} \sum_{\substack{0< rn-sm < r\sqrt{m}\\\gcd(n,m)=1}} \frac{m^2-(rn-ms)^2}{mr(rn-ms)}  &  \\
     \geq & \sum_{\substack{1\leq r< m^{1/4}\\ \gcd(m,r)=1}}\sum_{\substack{1\leq d<r\sqrt{m}\\ \gcd(d,mr)=1}} \frac{m^2-d^2}{mrd} & \text{by Lemma \ref{lem:boundonHandF} (1)}  %\displaybreak 
     \\ \geq & m(\sum_{\substack{1\leq r< m^{1/4}\\ \gcd(m,r)=1}} \sum_{\substack{1\leq d<r\sqrt{m}\\ \gcd(d,mr)=1}} \frac{1}{rd})-\sum_{\substack{1\leq r< m^{1/4}} }\sum_{\substack{1\leq d<r\sqrt{m}}} \frac{d}{mr}\\\geq & m\sum_{\substack{1\leq r< m^{1/4}\\ \gcd(m,r)=1}} \frac{\varphi(rm)}{rm}\frac{1}{r}\ln(r\sqrt m)-\Theta(\sqrt m) \\ \geq & \frac{1}2\varphi(m)\ln m\cdot\sum_{\substack{1\leq r< m^{1/4}\\ \gcd(m,r)=1}} \frac{\varphi(r)}{r^2}-\Theta(\sqrt m)   \\
     \geq & \frac{\varphi(m)^2}{2m}\ln m\cdot \sum_{\substack{1\leq r< m^{1/4}}} \frac{\varphi(r)}{r^2}-\Theta(\sqrt m)=\Theta\left(\tfrac{\varphi(m)^2}{m}(\ln m)^2\right).
          \end{align*}
          }
          The lower bound follows.
\end{proof}

\begin{Rem}\label{rem:lastwords}
    As we have mentioned in Remark \ref{rem:unstablesph}, it is not clear to us whether there exists a spherical bundle that is never slope semistable with respect to any polarization. We are also curious whether other kinds of `counting' are better than the naive one. For example, one may put different weights on spherical bundles. It is not clear to us whether one can build up any meaningful power series from these counting values.
\end{Rem}

\appendix

\section{Spherical twists and standard autoequivalences} \label{appendix}
\begin{center}
\small{by {\sc Genki Ouchi}}
\end{center}

In this appendix, we will prove that the spherical twists of spherical objects are not standard for smooth projective varieties of dimensions greater than one.
Let $X$ be a smooth projective $n$-dimensional variety over $\mathbb{C}$.

\begin{Def}
We define the {\it standard autoequivalence group} $A(X)$ as the subgroup
\[ A(X):=\mathbb{Z}[1] \times \left(\mathrm{Aut}(X) \ltimes \mathrm{Pic}(X) \right) \]
of the autoequivalence group $\mathrm{Aut}(\Db(X))$.
An autoequivalence $\Phi$ of $\Db(X)$ is called {\it standard with respect to} $X$ if $\Phi \in A(X)$ holds.
\end{Def}

The main theorem in Appendix A is as follows.

\begin{Thm}\label{thm:non-standard}
Assume that $n\geq 2$. Let $E$ be a $n$-spherical object in $\Db(X)$ and $k$ be a non-zero integer. Let $Y$ be a smooth projective variety with an exact equivalence $\Psi: \Db(X) \to \Db(Y)$.
There is not a standard autoequivalence $\Phi \in A(Y)$ such that $\mathsf{T}^k_E=\Psi^{-1} \circ \Phi \circ \Psi$ holds in $\mathrm{Aut}(\Db(X))$.
In particular, the autoequivalence $\mathsf{T}^k_E$ is not conjugate to a standard autoequivalence with respect to $X$ in $\mathrm{Aut}(\Db(X))$.
\end{Thm}

In the proof of Theorem \ref{thm:non-standard}, we use the notion of categorical entropy of autoequivalences of $\Db(X)$ introduced in \cite{DHKK}. 

Let $\Phi$ be an autoequivalence of $\Db(X)$. By \cite[Definition 2.4]{DHKK}, we have the categorical entropy $h_t(\Phi):\mathbb{R} \to [-\infty, \infty)$. By \cite[Definition 2.4, Proof of Lemma 2.5]{DHKK}, we have $h_0(\Phi) \in \mathbb{R}_{\geq 0}$. 
An object $G \in \Db(X)$ is called a {\it split generator} of $\Db(X)$ if $\Db(X)=\langle G\rangle_{\mathrm{thick}}$ holds, where $\langle G\rangle_{\mathrm{thick}}$ is the smallest thick subcategory of $\Db(X)$ containing $G$.
The categorical entropy $h_t(\Phi)$ is computed by the following formula.

\begin{Thm}[{\cite[Theorem 2.6]{DHKK}}]\label{thm:entropy}
\begin{itemize}
\item[$(1)$] For a split generator $G$ of $\Db(X)$, we have
\[ h_t(\Phi)=\lim_{N \to \infty}\frac{1}{N} \log \left( \sum_{m \in \mathbb{Z}} \hom(G, \Phi^N(G)[m])e^{-mt} \right)\]
for any $t \in \mathbb{R}$.
\item[$(2)$] Let $\ell$ be an integer. Then we have
\[ h_t(\Phi[\ell])=h_t(\Phi)+\ell t \]
for any $t \in \mathbb{R}$.
\item[(3)] For a positive integer $k$, we have
\[ h_t(\Phi^k)=kh_t(\Phi). \]
\end{itemize}
\end{Thm}

In Theorem \ref{thm:entropy}, the statements $(2)$ and $(3)$ are easily deduced from $(1)$.
See also \cite[Lem 2.7 (v)]{Kik} and \cite[Section 2.2]{DHKK}. By \cite[Lem 2.9]{Kik}, if autoequivalences $\Phi_1, \Phi_2$ of $\Db(X)$ are conjugate in $\mathrm{Aut}\Db(X)$, we have $h_t(\Phi_1)=h_t(\Phi_2)$ for any $t \in \mathbb{R}$. To prove Theorem \ref{thm:non-standard}, we compare $h_t(\mathsf{T}_E)$ with $h_t(\Phi)$ for a standard autoequivalence $\Phi \in A(Y)$, where $Y$ is a Fourier--Mukai partner of $X$. 
The following is the explicit formula of $h_t(\mathsf{T}_E)$. 

\begin{Thm}\label{thm:spherical_entropy}
The equality
\[ h_t(\mathsf{T}_E)=\left\{
\begin{array}{ll}
0 & (t\geq0) \\
(1-n)t & (t < 0)
\end{array}
\right.\]
holds.
\end{Thm}
\begin{proof}
This is a direct consequence of \cite[Theorem 1.4]{genkientropy} and Proposition \ref{thm:mainEperp}.
\end{proof}

Next, we compute the categorical entropy of standard autoequivalences.
Let $Y$ be a smooth projective variety with an exact equivalence $\Psi: \Db(X) \to \Db(Y)$. 
Take $\Phi \in A(Y)$. There is an automorphism $f$ of $Y$, a line bundle $\mathcal{L}$ on $Y$ and an integer $\ell$ such that 
$ \Phi= f^* \circ (\otimes \mathcal{L})[\ell]$. Note that $h_t(\Psi^{-1} \circ \Phi \circ \Psi)=h_t(\Phi)$ holds for any $t \in \mathbb{R}$. We have the following formula.

\begin{Prop}\label{prop:standard_entropy}
The equality 
\[ h_t(\Phi)=h_0(f^* \circ (\otimes \mathcal{L}))+\ell t \]
holds for any $t \in \mathbb{R}$.
\end{Prop}
\begin{proof}
Fix $t \in \mathbb{R}$.
By Theorem \ref{thm:entropy} (2), we have 
\[ h_t(\Phi)=h_t(f^* \circ (\otimes \mathcal{L}))+\ell t. \]

Consider the object
\[ G:=\bigoplus_{i=0}^n\mathcal{O}_Y(i),\]
where $\mathcal{O}_Y(1)$ is a very ample line bundle on $Y$. 
Then $G$ is a split generator of $\Db(Y)$. 
For any positive integer $N$, the objects $G$ and $(f^* \circ (\otimes \mathcal{L}))^N(G)$ are coherent sheaves on $Y$. By 
\cite[Lemma 2.10]{DHKK}, the categorical entropy $h_t(f^* \circ (\otimes \mathcal{L}))$ is constant with respect to $t$. Hence, we obtain 
\[ h_t(f^* \circ (\otimes \mathcal{L}))=h_0(f^* \circ (\otimes \mathcal{L})). \]
\end{proof}

Now, we give the proof of Theorem \ref{thm:non-standard}.
It is enough to prove the case of $k>0$.
Assume that there is a standard autoequivalence $\Phi \in A(Y)$ such that $\mathsf{T}^k_E=\Psi^{-1} \circ \Phi \circ \Psi$ holds in $\mathrm{Aut}(\Db(X))$. By Theorem \ref{thm:entropy} (3), we have  $kh_t(\mathsf{T}_E)=h_t(\mathsf{T}^k_E)=h_t(\Phi)$
for any $t \in \mathbb{R}$. By Theorem \ref{thm:spherical_entropy} and Proposition \ref{prop:standard_entropy}, this is a contradiction since $n \geq 2$. 

\begin{Rem}
Let $C$ be a smooth projective curve.
For a closed point $x \in C$, the skyscraper sheaf $\mathcal{O}_x$ is a $1$-spherical object in $\Db(C)$. As \cite[Example 8.10, i)]{Huybrechts:FM}, we have $\mathsf{T}_{\mathcal{O}_x} \simeq \otimes \mathcal{O}_C(x)$. In particular, the functor $\mathsf{T}_{\mathcal{O}_x} \in A(C)$ holds.
\end{Rem}

\bibliography{all}                      % .bib-Datei
\bibliographystyle{halpha}  
\end{document}